\documentclass[oneside]{amsart}

\usepackage[runin]{abstract}

\usepackage{geometry}                
\usepackage[utf8]{inputenc}

\usepackage{url}
\usepackage[foot]{amsaddr}
\usepackage{graphicx}
\usepackage{amssymb}

\usepackage{nicefrac}
\usepackage{enumerate}

\usepackage{hyperref}

\theoremstyle{plain} 
\newtheorem{thm}{Theorem}
\newtheorem{thm-int}{Theorem}

\newtheorem{lem}[thm]{Lemma}
\newtheorem{cor}[thm]{Corollary}
\newtheorem{prop}[thm]{Proposition}

\theoremstyle{definition} 
\newtheorem{defn}[thm]{Definition}	
\newtheorem{example}[thm]{Example}
\newtheorem{rmk}[thm]{Remark}

\def\real{\mathbb R}
\def\nat{\mathbb N}
\def\integ{\mathbb Z}

\def\metr{\mathbf{g}}
\def\bigo{\mathcal O}
\DeclareMathOperator\ord{ord}
\def\eps{\varepsilon}
\def\crex {\overrightarrow{\mathrm{exp}}\int}
\DeclareMathOperator\supp{supp}

\DeclareMathOperator\xLone{L^1}
\DeclareMathOperator\xCinfty{C^\infty}

\def\distr{\Delta}
\def\lcont{\xLone([0,T],\real^m)}
\def\lcontr{\xLone([0,T])}

\def\drift{f_0}
\DeclareMathOperator{\dist}{dist}
\def\dsr{d_{\mathrm{SR}}}

\def\sM{s}
\def\inter{I}
\def\epmtime{\mathcal T}
\newcommand{\adj}[2]{(\text{ad}^{#1}{#2})}

\newcommand{\td}[2]{f^{#1}_{#2}}

\newcommand{\ntd}[2]{\widehat f^{#1}_{#2}}
\newcommand{\homvf}[2]{F^{#1}_{#2}}
\def\dtd{\rho}
\def\dd{\rho^{\drift}}
\def\ddt{\rho_T^{\drift}}

\newcommand{\bsr}[2]{B_{\mathrm{SR}}(#1,#2)}
\newcommand{\btd}[2]{\mathcal R(#1,#2)}
\newcommand{\btdt}[3]{\mathcal R_{#3}(#1,#2)}
\newcommand{\bd}[2]{\mathcal R^{\drift}(#1,#2)}
\newcommand{\bdt}[2]{\mathcal R^{\drift}_T(#1,#2)}
\newcommand{\bbox}[1]{\mathrm{Box}\left(#1\right)}
\newcommand{\boxd}[1]{\Pi\left(#1\right)}

\newcommand{\vett}{\{f_1,\ldots,f_m\}}
\DeclareMathOperator{\Span}{span}
\DeclareMathOperator{\len}{length}
\def\cost{c}
\def\costd{c_{\drift}}

\newcommand{\app}{homogeneous series approximation }
\newcommand{\appsys}{series homogeneous}

\makeindex

\begin{document}

\title{H\"older continuity of the value function for control-affine systems}
	\thanks{This work was  supported by the European Research Council, ERC StG 2009 “GeCoMethods”, contract number 239748, and by the ANR project {\it GCM}, program ``Blanche'', project number NT09\_504490}
\author{Dario Prandi$^\dagger$$^\ddagger$}
    \address{$^\dagger$Centre National de Recherche Scientifique (CNRS), CMAP, \'Ecole Polytechnique, Route de Saclay, 91128 Palaiseau Cedex, France and Team GECO, INRIA-Centre de Recherche Saclay,}
    \email{prandi@cmap.polytechnique.fr}
    \address{$^\ddagger$SISSA, Trieste, Italy}

\maketitle

\begin{abstract}
	We prove the continuity and we give a Holder estimate for the value function associated with the $L^1$ cost of the control-affine system $\dot q = \drift(q)+\sum_{j=1}^m u_j f_j(q)$, satisfying the strong H\"ormander condition.
	This is done by proving a result in the same spirit as the Ball-Box theorem for driftless (or sub-Riemannian) systems.
	The techniques used are based on a reduction of the control-affine system to a linear but time-dependent one, for which we are able to define a generalization of the nilpotent approximation.
	Finally, we also prove the continuity of the value function associated with the $L^1$ cost of time-dependent systems of the form $\dot q = \sum_{j=1}^m u_j f_j^t(q)$. 
\end{abstract}



\section{Introduction}

	A sub-Riemannian control system on a smooth manifold $M$ is a control system in the form
	\begin{equation}\label{int:sr}
		\dot \gamma(t) = \sum_{i=1}^m u_i(t) \,f_i(\gamma(t)),\qquad \text{a.e. }t\in[0,T],
	\end{equation}
	where $u:[0,T]\rightarrow\real^m$ is an integrable control function and $\vett$ is a family of smooth vector fields satisfying the H\"ormander condition, i.e., such that its iterated Lie brackets generate the whole tangent space at any point. 
	The length of a curve $\gamma$ solving \eqref{int:sr}, is then defined as the $\len(\gamma)=\min\int_0^T |u|\,dt$, where the minimum is taken over all the possible $u(\cdot)$ satisfying the above ODE. Due to the linearity of the system w.r.t. $u$, this length will be independent of the parametrization of $\gamma$.
	Finally, we define
	\[
		\dsr(q,q') =\inf\{ \len(\gamma)\colon\: \gamma:[0,1]\to M \text{ is a solution of } \eqref{int:sr},\,\gamma(0)=q\text{ and } \gamma(T)=q'\}.
	\]
	By the H\"ormander condition, $\dsr$ is a distance, called Carnot-Carath\'eodory distance, endowing $M$ with a natural metric space structure.
	A manifold considered together with a sub-Riemannian control system is  called a sub-Riemannian manifold.
	
	Define $\distr^1=\Span\vett$ and $\distr^{s+1}=\distr^s+[\distr^{s},\distr^{1}]$, for every $s\in\nat$. Under the hypothesis that $\vett$ is equiregular, i.e., that, for each $s\in\nat$, the dimension of $\distr^s(q)$ is independent of $q\in M$, the H\"ormander condition implies that there exists a (minimal) $r\in\nat$ such that $\distr^r(q)=T_q M$ for all $q\in M$. Such $r$ is called the \em degree of non-holonomy \em of the sub-Riemannian control system.
	
	A fundamental result in the theory of sub-Riemannian manifolds is the celebrated Ball-Box theorem (see for example \cite{Bellaiche1996}). This theorem gives a rough description of the infinitesimal shape of the sub-Riemannian balls. Namely, at any point $q$ of an equiregular sub-Riemannian manifold, the sub-Riemannian ball of small radius $\eps$ is equivalent, in privileged coordinates, to the box
	\[
		\underbrace{[-\eps,\eps]\times\ldots\times[-\eps,\eps]}_{\dim\distr^1} \times \ldots\times \underbrace{[-\eps^{s},\eps^{s}]\times\ldots\times [-\eps^{s},\eps^{s}]}_{\dim\distr^{s}-\dim\distr^{s-1}}\times \ldots \times\underbrace{[-\eps^r,\eps^r]\times \ldots\times [-\eps^r,\eps^r]}_{\dim\distr^{r}-\dim\distr^{r-1}}.
	\]
	By this we mean that, for some constant $C>0$, the sub-Riemannian ball is contained in a box of side $C\eps$, and contains a box of side $\eps/C$.
	
	This fact has a plethora of applications. 
	First, it allows to prove a H\"older regularity estimate with respect to the Euclidean distance in coordinates, namely that, for $q'$ sufficiently close to $q$, it holds
	\begin{equation}\label{int:holder}
		|q'|\lesssim \dsr(q,q') \lesssim |q'|^{\nicefrac{1}{r}}.
	\end{equation}
	Here we use ``$\lesssim$" to denote an inequality up to a multiplicative constant, independent of $q'$.
	Then, among many others, it is a fundamental step in the computation of the Hausdorff dimension of the manifold (see \cite{Mitchell1985}), and it is used to obtain asymptotic estimates on the heat kernel (see e.g., \cite{Rothschild1976,Jerison1987,Folland1974,Agrachev2009}). Moreover, it is the main tool in computing the asymptotic equivalents of the entropy and the complexity of curves (see e.g., \cite{Jean2001,Jean2003a,Romero-Melendez2004,Gauthier2005,Gauthier2006}).
	
	In this paper, we focus on a very important generalization of the control system \eqref{int:sr}, namely on control-affine systems. These systems are obtained by adding to \eqref{int:sr} an uncontrolled vector field $\drift$, called the \em drift\em,  and are in the form
	\begin{equation}\label{int:d}
		\dot \gamma(t) = f_0(\gamma(t))+ \sum_{i=1}^m u_i(t) \,f_i(\gamma(t)),\qquad \text{a.e. }t\in[0,T].
	\end{equation}
	These kind of systems appears in plenty of applications. As an example we cite, mechanical systems with controls on the acceleration (see e.g., \cite{Bullo2004}, \cite{Sigalotti2010}), where the drift is the velocity, or quantum control (see e.g., \cite{D'Alessandro2008}, \cite{Boscain2006}), where the drift is the free Hamiltonian.
	We always assume the strong H\"ormander condition, i.e., that the family $\vett$ satisfies the H\"ormander condition.
	
	The cost of a curve $\gamma$ solving \eqref{int:d} is $\costd(\gamma)=\min\int_0^T|u|\,dt$. 
	Unlike the sub-Riemannian length, due to the presence of the drift, the cost depends on the parametrization of the curve. 
	Finally, the value function, between $q,q'\in M$, of the control system at time $T>0$, is defined as
	\[
		\ddt(q,q')=\inf\{ \costd(\gamma)\colon\: \gamma:[0,T']\to M \text{ solves \eqref{int:d},  } \gamma(0)=q, \, \gamma(T')=q'\text{, and }T'\le T \}.
	\]
	Assume now that the drift is regular, in the sense that there exists $s\in\nat$ such that  $\drift(q)\in\distr^s(q)\setminus \distr^{s-1}(q)$, for any $q\in M$, where $\distr^s$ is defined through the vector fields $\vett$ as before.
	Our main result is, then, a generalization of \eqref{int:holder} to this context. 
	\begin{thm}\label{thm:dhold}
		Let $z=(z_1,\ldots,z_n)$ be a system of privileged coordinates at $q$ for $\vett$, rectifying $\drift$ as the $k$-th coordinate vector field $\partial_{z_k}$, for some $1\le k \le n$. 
		Then, for sufficiently small $T$ and $\eps$, if $\ddt(q,q')\le \eps$ then in $z$ coordinates it holds
		\[
			\dist\left(z(q'),z(e^{[0,T]\drift}q)\right) \lesssim \,\ddt (q,q')\, \lesssim  \dist\left(z(q'),z(e^{[0,T]\drift}q)\right)  ^{\nicefrac{1}{r}}.
		\]
		Here for  any $x\in\real^n$ and $A\subset\real^n$,  $\dist(x,A)=\inf_{y\in A}|x-y|$ denotes the Euclidean distance between them, $e^{t\drift}$ denotes the flow of $\drift$, and $r$ is the degree of non-holonomy of the sub-Riemannian control system defined by $\vett$.
	\end{thm}
	In this result, instead of the Euclidean distance from the origin that appeared in \eqref{int:holder}, we have the distance from the integral curve of the drift. This is due to the fact that moving in this direction has null cost. 
	As in the sub-Riemannian case, Theorem~\ref{thm:dhold} is a consequence of an estimate on the shape of the reachable sets, contained in Theorem~\ref{thm:dballbox}. 
	Moreover, although Theorem~\ref{thm:dhold} seems a natural generalization of \eqref{int:holder}, the shape of the reachable sets described in Theorem~\ref{thm:dballbox} is much more complicated than the boxes of the sub-Riemannian case, yielding a more difficult proof. 
	Theorem~\ref{thm:dhold} and \ref{thm:dballbox} represent the key step for generalizing the estimates on the complexity of curves from sub-Riemannian control systems to control-affine systems.
	
	It is worth to mention that these results regarding control-affine systems are obtained by reducing them, as in \cite{Agrachev2010a}, to time-dependent control systems in the form
	\begin{equation}\label{int:td}
		\dot \gamma(t) = \sum_{i=1}^m u_i(t) \,f_i^t(\gamma(t)),\qquad \text{a.e. }t\in[0,T],
	\end{equation}
	where $f^t_i=(e^{-t\drift})_* f_i\,$ is the pull-back of $f_i$ through the flow of the drift. 
	On these systems, that are linear in the control, we are able to define a good notion of approximation of the control vector fields. Namely, in Section~\ref{subsec:etd} we will define a generalization of the nilpotent approximation, used in the sub-Riemannian context, taking into account the fact that in the system \eqref{int:td}, exploiting the time, we can generate the direction of the brackets between $\drift$ and the $f_j$s. 
	This approximation and an iterated integral method yield the correct estimates on the reachable set, contained in Theorem~\ref{thm:ballbox}.
	
	\vspace{.5cm}
	The paper is divided in three sections. In Section~\ref{sec:sr} we recall some generalities and definitions regarding sub-Riemannian control systems, used in the following sections. 
	In Section~\ref{sec:td} we consider control systems in the form \eqref{int:td}, and we prove the continuity of the value function  for general time-dependent vector fields. Then, in Theorem~\ref{thm:ballbox}, restricting then to the case where the time dependency is explicitly given as $f_i^t=(e^{-t\drift})_* f_i$, we establish some estimates on the reachable sets, in the same spirit as the Ball-Box theorem. 
	Finally, in Section~\ref{sec:drift} we consider control-affine systems. After proving the relation between control-affine systems and time-dependent systems, we prove the continuity of the value function. Then, in Lemma~\ref{lem:time}, exploiting the affine nature of the control system, we give an upper bound on the time needed to join two points $q$ and $q'$ as a function of $\ddt(q,q')$. From this fact and the estimates of Section~\ref{sec:td}, Theorems~\ref{thm:holder}~and~\ref{thm:dballbox} follow. Theorem~\ref{thm:dhold} is then a particular case of Theorem~\ref{thm:holder}, that holds under slightly milder assumptions on $\drift$ and $\vett$.
	
\section{Sub-Riemannian Geometry}\label{sec:sr}

Throughout this paper, $M$ is an $n$-dimensional connected smooth manifold. 
In this section we recall some classical notions and results of sub-Riemannian geometry.

\subsection{Sub-Riemannian control systems}
A sub-Riemannian (or non-holonomic) control system on $M$ is a control system in the form
	\begin{equation}\label{cs:sr}\tag{SR}
		\dot q = \sum_{i=1}^m u_i \,f_i(q),\qquad q\in M,\quad u=(u_1,\ldots,u_m)\in\real^m,
	\end{equation}
	where $\vett$ is a family of smooth vector fields on $M$. We let $f_u=\sum_{i=1}^mu_i\,f_i$.

	An absolutely continuous curve $\gamma:[0,T]\to M$ is \em \eqref{cs:sr}-admissible \em if there exists a control $u\in \lcont$ such that $\dot\gamma(t)=f_{u(t)}(\gamma(t))$, for a.e. $t\in [0,T]$. 
	The curve is said to be \em associated \em to any such control.
	The length of $\gamma$ is defined as 
	\begin{equation}\label{def:length}
		\len(\gamma)= 
		\min  \|u\|_{\lcont},
	\end{equation}
	where the minimum is taken over all controls $u$ such that  $\gamma$ is associated with $u$.
	It is attained, due to convexity reasons.  
	Notice that, by definition, $\len(\gamma)$ is invariant under time reparametrization of the curve. 
	The \em distance \em induced by the sub-Riemannian system on $M$ is then defined as
	\[
	\begin{split}
		\dsr(q,q') &=\inf\{ \len(\gamma)\colon\: \gamma \text{ \eqref{cs:sr}-admissible and } \gamma:q\rightsquigarrow q'\},
	\end{split}
	\]
	where $\gamma:q\rightsquigarrow q'$ stands for $\gamma:[0,T]\to M$, for some $T>0$,  $\gamma(0)=q$ and $\gamma(T)=q'$. 

	Let $\distr$ be the $\xCinfty$-module generated by the vector fields $\vett$ (in particular, it is closed under multiplication by $\xCinfty(M)$ functions and summation). Let $\distr^1=\distr$, and define recursively $\distr^{s+1}= \distr^s+[\distr^s,\distr]$, for every $s\in\nat$. Due to the Jacobi identity $\distr^s$ is the $\xCinfty$-module of linear combinations of all commutators of $f_1,\ldots,f_m$ with length $\le s$. 
	For $q\in M$, let $\distr^s(q)=\{ f(q)\colon\: f\in\distr^s\}\subset T_q M$. We say that $\{f_1,\ldots,f_m\}$ satisfies the \em H\"ormander condition \em (or that it is a bracket-generating family of vector fields) if $\bigcup_{s\ge1} \distr^s(q)=T_q M$ for any $q\in M$. 
In the following we will always assume this condition to be satisfied. 

By the Chow--Rashevsky theorem (see for instance \cite{Agrachev2010}), the hypothesis of connectedness of $M$ and the H\"ormander condition guarantee the finiteness and  continuity of $\dsr$ with respect to the topology of $M$. 
Hence, the function $\dsr$, called \em sub-Riemannian \em or \em Carnot-Carath\'eodory distance\em, induces on $M$ a metric space structure. The open balls of radius $\eps>0$ and centered at $q\in M$, with respect to $\dsr$, are denoted by $\bsr q \eps$. 

	We say that a \eqref{cs:sr}-admissible curve $\gamma$ is a minimizer of the sub-Riemannian distance between $q,q'\in M$ if $\gamma:q\rightsquigarrow q'$ and $\len(\gamma)=\dsr(q,q')$. Equivalently, $\gamma$ is a minimizer between $q,q'\in M$ if it is a solution of the free-time optimal control problem, associated with \eqref{cs:sr},
	\begin{equation}\label{eq:l1}
		\|u\|_{\xLone(0,T)}=\int_0^T \sqrt{\sum_{j=1}^m u_j^2(t)}\,dt\rightarrow \min,\qquad \gamma(0)=q,\quad \gamma(T)=q',\quad T>0.
	\end{equation}
	Indeed, the sub-Riemannian distance is the value function associated with this problem.
	It is a classical result that, for any couple of points $q,q'\in M$ sufficiently close, there exists at least one minimizer. 
	
\begin{rmk}\label{rmk:contr}
	The optimal control problem \eqref{eq:l1} is equivalent to the following, with $p\ge1$ and $T>0$ fixed,
	\begin{equation}\label{eq:lp}
		\|u\|_{L^p(0,T)}=\left( \int_0^T |u|^p\,dt \right)^{\nicefrac{1}{p}} \rightarrow \min,\qquad \gamma(0)=q,\quad \gamma(T)=q',
	\end{equation}
	In fact, due to the invariance under time reparametrization of system \eqref{cs:sr}, in \eqref{eq:l1} we can fix either $T>0$ or the Euclidean norm of $u$.
	Moreover, by the H\"older inequality, for any $p> 1$, letting $p'$ be the conjugated exponent  to $p$ (i.e., $1/p+1/p'=1$), we get
	$
		\|u\|_{\xLone(0,T)}\le T^{\nicefrac{1}{p'}} \|u\|_{L^p(0,T)},
	$
	with the equality holding if and only if $|u|$ is constant. 
	From these two facts, it is easy to check that minimizers of the optimal control problem \eqref{eq:lp} coincide with the minimizers of \eqref{eq:l1} with constant norm. Thus the two optimal control problems are equivalent. 
	\end{rmk}
	
	\begin{rmk}
	This control theoretical setting can be stated in purely geometric terms. Indeed, it is equivalent to a \em generalized sub-Riemannian \em structure. Such a structure is defined by a rank-varying smooth distribution and a Riemannian metric on it (see \cite{Agrachev2010} for a precise definition). 	
	In a sub-Riemannian control system, in fact, the map $q\mapsto \Span\{f_1(q),\ldots,f_m(q)\}\subset T_qM$ defines a rank-varying smooth distribution, which is naturally endowed with the Riemannian norm defined, for $v\in \distr(q)$, by
	\[
		\metr(q,v)=\inf \left\{ |u|=\sqrt{ u_1^2+\cdots+u_m^2} \colon\: f_u (q)= v \right\}.
	\] 
	The pair $(\distr,\metr)$ is thus a generalized sub-Riemannian structure on $M$. Conversely, every rank-varying distribution is finitely generated, see \cite{Agrachev2008,Agrachev2010,Agrachev2010b,Drager2011}, and thus a sub-Riemannian distance can be written, globally, as the value function of a control system of the type \eqref{cs:sr}.
	\end{rmk}
	
	\subsection{Privileged coordinates and nilpotent approximation}
	
	We now introduce the equivalent, in the sub-Riemannian context, of the linearization of a vector field. This classical procedure, called \em nilpotent approximation\em, is possible only in a carefully chosen set of coordinates, called \em privileged coordinates\em.
			
	Since $\{f_1,\ldots,f_m\}$ is bracket-generating, the values of the sets $\distr^s$ at $q$ form a flag of subspaces of $T_q M$,
	\[
		\distr^1(q)\subset \distr^2(q) \subset\ldots \subset \distr^{r}(q)=T_qM.
	\]
	The integer $r=r(q)$, which is the minimum number of brackets required to recover the whole $T_qM$, is called \em degree of non-holonomy \em (or \em step\em) of the family $\vett$ at $q$. 
	Set $n_s(q)=\dim \distr^s(q)$. The integer list $(n_1(q),\ldots,n_{r}(q))$ is called the \em growth vector \em at $q$. 
	From now on we fix $q\in M$, and denote by $r$ and $(n_1,\ldots,n_r)$ its degree of non-holonomy and its growth vector, respectively.
	Finally, let $w_1\le\ldots\le w_n$ be the \em weights \em associated with the flag, defined by $w_i=s$ if $n_{s-1}<i\le n_s$, setting $n_0=0$. 
	 
	For any smooth vector field $f$, we denote its action, as a derivation on smooth functions, by $f: a\in \xCinfty(M)\mapsto fa\in \xCinfty(M)$. For any smooth function $a$ and every vector field $f$ with $f\not\equiv 0$ near $q$, their \em (non-holonomic) order \em at $q$ is
	\[
	\begin{split}
		\ord_{q}(a)&=\min\{ s\in\nat\colon\: \exists i_1,\ldots,i_s\in\{1,\ldots,m\} \text{ s.t. } (f_{i_1}\ldots f_{i_s}\,a)(q)\neq 0\},\\	
		\ord_{q}(f)&=\max\{ \sigma\in\integ \colon\: \ord_{q}(fa) \ge \sigma+\ord_{q}(a) \text{ for any } a\in \xCinfty(M)\}.
	\end{split}
	\]
	In particular it can be proved that $\ord_q(a)\ge s$ if and only if $a(q')=\bigo(\dsr(q',q))^s$.	 
	 	 
	\begin{defn}
		A \em system of privileged coordinates \em at $q$ for $\vett$ is a system of local coordinates $z=(z_1,\ldots,z_n)$ centered at $q$ and such that $\ord_q(z_i)=w_i$, $1\le i\le n$.
	\end{defn}

	For any point $q\in M$ there always exists a system of privileged coordinates around $q$.  
	Consider such a system $z=(z_1,\ldots,z_n)$. 
	We now show that this allows to compute the order of functions or vector fields in a purely algebraic way. Given a multiindex $\alpha=(\alpha_1,\ldots,\alpha_n)$ we define the weighted degree of the monomial $z^\alpha=z_1^{\alpha_1}\cdots z_n^{\alpha_n}$ as $w(\alpha)=w_1\alpha_1+\cdots+w_n\alpha_n$ and the weighted degree of the monomial vector field $z^\alpha\partial_{z_j}$ as $w(\alpha)-w_j$. Then one can prove that, given $a\in \xCinfty(M)$ and a smooth vector field $f$, with Taylor expansion
	\[
		a(z)\sim \sum_{\alpha} a_\alpha z^\alpha \quad \text{ and }\quad f(z)\sim \sum_{\alpha,j} f_{\alpha,j} z^\alpha\partial_{z_j},
	\]
	their orders at $q$ can be computed as
	\[
		\ord_q(a)=\min\{ w(\alpha)\colon\: a_\alpha\neq 0\}\quad \text{ and }\quad
		\ord_q(f)=\min\{ w(\alpha)-w_j\colon\: f_{\alpha,j} \neq 0\}.
	\]	
	A function or a vector field are said to be \em homogeneous \em if all the nonzero terms of the expansion have the same weighted degree.
	
	 We recall that, for any $a,b\in \xCinfty(M)$ and any smooth vector fields $f,g$, the order satisfies the following properties
	 \begin{equation}\label{eq:ord}
	 \begin{split}
		\ord_q(a+b)  =\min\{\ord_q(a),\ord_q(b)\},\qquad &\ord_q(ab) = \ord_q(a)+\ord_q(b),\\
		\ord_q(f+ g) =\min\{\ord_q(f),\ord_q(g)\}, \qquad &\ord_q([f,g]) \ge \ord_q(f)+\ord_q(g).  \\
	\end{split}
	 \end{equation}
	
	Consider $f_i$, $1\le i\le m$. 
	By the definition of order, it follows that $\ord_q(f_i) \ge -1$. Then we can express $f_i$ in coordinates as
	\[
		z_* f_i=\sum_{j=1}^n\big( h_{ij}+r_{ij} \big)\partial_{z_j},
	\]
	where $z_*$ is the push-forward operator on vector fields associated with the coordinates, defined as $z_*f=dz\circ f\circ z^{-1}$, $h_{ij}$ are homogeneous polynomials of weighted degree $w_j-1$, and $r_{ij}$ are  functions of order larger than or equal to $w_j$. 
	
	\begin{defn}\label{def:nil}
		The nilpotent approximation at $q$ of $f_i$, $1\le i\le m$, associated with the privileged coordinates $z$ is the vector field with coordinate representation
		\[
			z_*\widehat f_i=\sum_{j=1}^n h_{ij} \,\partial_{z_j}.
		\]
	\end{defn}
	The nilpotentized sub-Riemannian control system is then defined as
	\begin{equation}\label{cs:nsr}\tag{NSR}
		\dot q = \sum_{j=1}^m u_j(t) \widehat f_j(q).\\
	\end{equation}
	The family of vector fields $\{\widehat f_1,\ldots,\widehat f_m\}$ is bracket-generating and nilpotent of step $r$ (i.e., every iterated bracket $[f_{i_1},[\ldots,[f_{i_{k-1}},f_{i_k}]]]$ of length larger than $r$ is zero). 
	
	The main consequence of the nilpotent approximation is the following (see for example \cite[Proposition 7.29]{Bellaiche1996}).
	\begin{prop}\label{prop:nilsr}
		Let $z=(z_1,\ldots,z_n)$ be a system of privileged coordinates at $q\in M$ for $\vett$. 
		For  $T>0$ and $u\in \xLone([0,T];\real^m)$, with $|u|\equiv 1$, let $\gamma(\cdot)$ and $\hat \gamma(\cdot)$ be the trajectories associated with $u$ in \eqref{cs:sr} and \eqref{cs:nsr}, respectively, and such that $\gamma(0)=\hat\gamma(0)=q$. 
	Then, there exist $C,T_0>0$, independent of $u$, such that, for any $t<T_0$, it holds 		
	\begin{equation}
		|z_i(\gamma(t))-z_i(\hat \gamma (t))|\le C t^{w_i+1},\qquad i=1,\ldots,n.
	\end{equation}
	\end{prop}

	We recall, finally, the celebrated Ball-Box Theorem, that gives a rough description of the shape of small sub-Riemannian balls.
		
	\begin{thm}[Ball-Box Theorem]\label{thm:srballbox} Let $z=(z_1,\ldots,z_n)$ be a system of privileged coordinates at $q\in M$ for $\vett$. Then there exist $C,\eps_0>0$ such that for any $\eps<\eps_0$, it holds
	\[
		\bbox {\frac 1 C \eps} \subset \bsr {q} \eps \subset \bbox {C \eps},
	\] 
	where, $\bsr q \eps$ is identified with its coordinate representation $z(\bsr q \eps)$ and,        for any $\eta>0$, we let
	\begin{equation}\label{def:box}
		\bbox \eta = \{z\in\real^n\colon\: |z_i|\le \eta^{w_i} \},
	\end{equation}

	\end{thm}	
	
	Observe that the first inclusion follows directly from the definition of privileged coordinates.
	
	As a corollary of the Ball-Box Theorem, we get the following result on the regularity of the distance.
	
	\begin{cor}\label{cor:holder}
		Let $z=(z_1,\ldots,z_n)$ be a system of privileged coordinates at $q\in M$ for $\vett$. Then there exists $C,\eps>0$ such that 
		\[
			\frac 1 C |z(q')|\le \dsr(q,q')\le C |z(q')|^{\nicefrac{1}{r}},\qquad q'\in \bsr {q} \eps.
		\]
	\end{cor}

	\section{Time-dependent systems}\label{sec:td}
	
	\subsection{Time-dependent control systems}
	We now consider a more general situation.
	Namely, we consider on $M$ the time-dependent non-holonomic control system
	\begin{equation}\label{cs:td}\tag{TD}
		\dot q=\sum_{i=1}^m u_i\, f^t_i(q),\qquad q\in M,\quad u=(u_1,\ldots,u_m)\in\real^m, \quad t\in \inter,
	\end{equation}
	where $\inter=[0,b)$ for some $b\le+\infty$ and $\{f_1^t,\ldots,f_m^t\}$ is a family of non-autonomous smooth vector fields, with smooth dependence on the time parameter. We let $f_u^t=\sum_{i=1}^m u_i\,f_i^t$.
	
	In analogy with the autonomous case, we define \eqref{cs:td}-admissible curves as absolutely continuous curves $\gamma:[0,T]\subset\inter\to M$ such that $\dot\gamma(t)=f_{u(t)}^t(\gamma(t))$ for a.e. $t\in[0,T]$, for some control $u\in \lcont$. Observe, however, that contrary to what happens in the sub-Riemannian case, the \eqref{cs:td}-admissibility property is not invariant under time reparametrization, e.g., a time reversal. Thus, we define the cost (and not the length) of $\gamma$ to be 
	\[
		\cost(\gamma)=\min \|u\|_{\lcont},
	\]
	where the minimum is taken over all controls $u$ such that  $\gamma$ is associated with $u$ and is attained due to convexity. 
	The \em value function \em induced by the time-dependent system is then defined as
	\[
		\dtd(q,q')=\inf\{ \cost(\gamma)\colon\: \gamma \text{ is \eqref{cs:td}-admissible and } \gamma:q\rightsquigarrow q' \}.
	\]
	Clearly, the value function is non-negative. It is not a metric since, in general, it fails both to be symmetric and to satisfy the triangular inequality. 
	Moreover, as the  following example shows,  
	$\dtd$ could be degenerate. Namely, it could happen that $q\neq q'$ but $\dtd(q,q')=0$.
	
	\begin{example}\label{ex:td}
	Let $M=\real$, with coordinate $x$ and consider the vector field $f^t= {(1-t)^{-2}} \partial_x$ defined on $[0,1)$. 
	For any $x_0\in\real$, $x_0\neq 0$, and for any sequence $t_n\uparrow 1$,  let $u_n\in \xLone([0,t_n])$ be defined as $u_n\equiv  ({1-t_n}) {x_0} $. 
	By definition, each $u_n$ steers the system from $0$ to $x_0$. Hence,
\[
	\rho_1(0,x_0)\le \inf_{n\in\nat} \| u_n\|_{\xLone([0,t_n])}=   \inf_{n\in\nat}  \int_0^{t_n}  ({1-t_n}){x_0}\,dt = x_0 \inf_{n\in\nat} t_n (1-t_n) =0.
\]
	This proves that, for any $x_0\in\real$, $\rho_1(0,x_0)=0$.
	\end{example}
 
	For $T>0$, $q\in M$ and $\eps>0$, we denote the reachable set from $q$ with cost less than $\eps$ by 
	\[
		\btd {q} \eps  =\{q'\in M\colon\: \dtd(q,q')<\eps\}.
	\]
	We will also consider the reachable set from ${q}$ in time less than $T>0$ and cost less than $\eps$, and denote it by $\btdt {q} \eps T$. Clearly $\btdt {q} \eps T \subset \btd {q} \eps$.
	
	In general, the existence of minimizers for the optimal control problem associated with \eqref{cs:td} is not guaranteed. 
	We conclude this section with an example of this fact.
			
	\begin{example}
	Let $M=\real$, with coordinate $x$, and consider the vector field $f^t=e^{-t} \partial_x$ for $t\in [0,1)$.
	Fix $x_0\in\real$, $x_0\neq0$. Observe that, for any $T>0$ and any control $u\in \lcontr$ steering the system from $0$ to $x_0$ , it holds
	\begin{equation}\label{ex:xz}
		|x_0| = \left|\int_0^T u(t)e^{-t}\,dt \right| \le\int_0^T |u(t)|e^{-t}\,dt < \|u\|_{\lcontr}.
	\end{equation}
	This implies $\dtd(0,x_0)\ge |x_0|$. 
	Let now $u_n\in \xLone([0,1/n])$ be defined as $u_n(t)=x_0 n e^t$. Clearly $u_n$ steers the system from $0$ to $x_0$. Moreover,
	\[
		\dtd(0,x_0)\le \inf_{n\in\nat} \|u_n\|_{\xLone([0,1/n])} = |x_0| \inf_{n\in\nat} \frac{e^{\frac 1 n}-1}{\frac 1 n} = |x_0|.
	\]
	This proves that $\dtd(0,x_0) = |x_0|$. Hence, the non-existence of minimizers follows from \eqref{ex:xz}. 

	\end{example}
	
	\subsection{Finiteness and continuity of the value function}
	In this section, we extend the Chow--Rashevsky Theorem to time-dependent non-holonomic systems, under the strong H\"ormander condition, whose definition follows.  
	
	\begin{defn}\label{def:hor}
		We say that a family of time-dependent vector fields $\{\td t 1,\ldots, \td t m\}_{t\in \inter}$ satisfies the \em strong H\"ormander condition \em if $\{\td {t_0} 1,\ldots, \td {t_0} m\}$ satisfies the H\"ormander condition for any $t_0\in \inter$.
	\end{defn}
		
	As we will see later on in Section~\ref{sec:drift}, when considering families of time-dependent vector fields of the form $\td t i = (e^{-t\drift})_*f_i$ this condition is equivalent to the strong H\"ormander condition for the affine control system with drift $\drift$ and control vector fields $\vett$. 

    From now on we will assume that the following holds.
    \begin{equation}\tag{H$_{0}$}\label{hyp:shc}
    	\begin{split}
    		&\text{The family of smooth vector fields }\{\td t 1,\ldots, \td t m\}_{t\in \inter},\, \text{depends smoothly on } t \\
    		&\text{and satisfies the strong H\"ormander condition.}
    	\end{split}
    \end{equation}

	This section will be devoted to the proof of the following.
		
	\begin{thm}\label{prop:dsrp}
	Assume that $\{\td t 1,\ldots, \td t m\}_{t\in \inter}$ satisfies \eqref{hyp:shc}. Then, the function $\dtd:M\times M\rightarrow [0,+\infty)$ is continuous. Moreover, for any $t_0\in\inter$ and any $q,q'\in M$, letting $\dsr$ be the sub-Riemannian distance induced by $\{f_1^{t_0},\ldots,f_m^{t_0}\}$, it holds $\dtd(q,q')\le \dsr(q,q')$.
	\end{thm}
	
	Now, we need to introduce some notation. 
	Following \cite{Sachkov2004}, the flows between times $s,t\in\real$ of an autonomous vector field $f$ and of a non-autonomous vector field $\tau\mapsto f^\tau$ will be denoted by, respectively,
	\[
		e^{(t-s)f}:M\to M\quad\text{and} \quad \crex_s^t f^\tau\,d\tau: M\to M.
	\]

	Fix $q\in M$ and assume, for the moment, that $t_0=0$. Let $\ell\in\nat$ and $\mathcal F=({i_1},\ldots,{i_\ell})\in\{1,\ldots,m\}^\ell$.
	For  any $\epmtime\in\inter$, $\epmtime>0$, we define the \em switching end-point  map \em at time $\epmtime$ and associated with $\mathcal F$ to be the function $E_{\epmtime,\mathcal F}:\real^\ell \rightarrow M$ defined as 
	\begin{equation}\label{eq:epm}
		\begin{split}
		E_{\epmtime,\mathcal F}(\xi)&=\crex_{\frac{\ell-1}{\ell} \epmtime}^\epmtime\frac{1} \epmtime \xi_\ell\, f^\tau_{i_\ell}\,d\tau \circ \cdots \circ \crex_0^{\frac \epmtime \ell} \frac{1} \epmtime\xi_1\, f_{i_1}^\tau\,d\tau \,(q)\\
		&=\crex_{\frac{\ell-1}{\ell}}^1  \xi_\ell\, f^{{\tau}{\epmtime}}_{i_\ell}\,d\tau \circ \cdots \circ \crex_0^{\frac 1 \ell} \xi_1\, f_{i_1}^{{\tau}{\epmtime}}\,d\tau \,(q).\\
		\end{split}
	\end{equation}
	Here we applied a standard change of variables formula for non-autonomous flows. 
	Let then
	\begin{equation}\label{eq:gtau}
		g^\tau_{\epmtime,\mathcal F} = 
		\begin{cases}
			\xi_1\, f_{i_1}^{\tau\epmtime} & \text{if } 0\le \tau<{1}/{\ell},\\
			\xi_2\, f_{i_2}^{(\tau-1/\ell) \epmtime} & \text{if } 1/\ell\le \tau<{2}/{\ell},\\
			\qquad \vdots\\
			\xi_\ell\, f_{i_\ell}^{(\tau-(\ell-1)/\ell) \epmtime} & \text{if } (\ell-1)/\ell\le \tau<{1},\\
		\end{cases}
	\end{equation}
	so that we can write
	\[	
		E_{\epmtime,\mathcal F}(\xi)=\crex_0^1 g_{\epmtime,\mathcal F}^\tau(\xi)\,d\tau\,(q).
	\]
	Clearly, $t\mapsto \crex_0^t g^\tau_{\epmtime,\mathcal F}(\xi)\,d\tau\,(q)$, $t\in[0,1]$, is a \eqref{cs:td}-admissible trajectory. Thus, $E_{\epmtime,\mathcal F}(\xi)$, $T>0$, is the end-point of a piecewise smooth \eqref{cs:td}-admissible curve. 
	
	We recall that, by the series expansion of $\overrightarrow{\text{exp}}$ (see \cite{Sachkov2004}), for any non-autonomous smooth vector field $f^\tau$, it holds $\crex_0^t f^\tau\,d\tau\,(q)=e^{t\, f^0}(q)+\bigo(t^2)$. Thus, we can define 
	\[
		E_{0,\mathcal F}(\xi)=\lim_{\epmtime\downarrow 0} E_{\epmtime,\mathcal F}(\xi)= e^{\xi_\ell\,f_\ell^0}\circ\ldots\circ e^{\xi_1\,f_1^0}(q)=\crex_0^1 {g^\tau_{0,\mathcal F}(\xi)}\,d\tau\,(q),
	\]
	where, $g^\tau_{0,\mathcal F}(\xi)$ is defined in \eqref{eq:gtau}.
	Then $t\mapsto \crex_0^t{g^\tau_{0,\mathcal F}(\xi)}\,d\tau\, (q)$, $t\in[0,1]$, is an \eqref{cs:sr}-admissible curve for the sub-Riemannian structure defined by $\{f_1^0,\ldots,f_m^0\}$ and $E_{0,\mathcal F}(\xi)$ is the end-point of a piecewise smooth trajectory in \eqref{cs:sr}. 
        
	After \cite{Sussmann1976}, we say that a point $q'\in M$ is \em \eqref{cs:td}-reachable \em from $q$ at time $t_0=0$, if there exist $\ell\in\nat$, $\mathcal F\in \{1,\ldots,m\}^\ell$, $\epmtime>0$ and $\xi\in\real^\ell$, such that $E_{\epmtime,\mathcal F}(\xi)=q'$. In this case it is clear that $\dtd(q,q')\le \sum_i |\xi_i|$. Moreover, if $\xi'\mapsto E_{\epmtime,\mathcal F}(\xi')$ has rank $n$ at $\xi$, the point $q'$ is said to be \em \eqref{cs:td}-normally reachable \em at time $t_0=0$.
    Finally, the point $q'$ is said to be \em \eqref{cs:sr}-reachable \em or \em \eqref{cs:sr}-normally reachable \em for the vector fields $\{f_1^0,\ldots,f_m^0\}$, if these properties holds for $\epmtime=0$.
        
    In the case $t_0> 0$, taking $\epmtime>0$ such that $\epmtime+t_0\in\inter$ and changing the interval of integration in \eqref{eq:epm} from $[0,\epmtime]$ to $[t_0,t_0+\epmtime]$, it is clear how to define \em \eqref{cs:td}-reachable \em and \em \eqref{cs:td}-normally reachable \em points from $q$ at time $t_0$, and \em \eqref{cs:sr}-reachable \em and \em \eqref{cs:sr}-normally reachable \em points for the vector fields $\{f_1^{t_0},\ldots,f_m^{t_0}\}$.
        
	The proof of the following lemma is an adaptation of \cite[Lemma 3.1]{Sussmann1976}. 
	
\begin{lem}\label{lem:suss}
      Let $q'\in M$ be \eqref{cs:sr}-normally reachable for the vector fields $\{f_1^{t_0},\ldots,f_m^{t_0}\}$ from $q$, by some $\ell\in\nat$, $\xi\in\real^\ell$ and $\mathcal F\in\{1,\ldots,m\}^\ell$. 
      Then, there exist $\eps_0,\epmtime_0>0$ such that, for any $\eps<\eps_0$, the point $q'$ is \eqref{cs:td}-normally reachable at time $t_0$, by the same $\ell$ and $\mathcal F$, and some  $\xi'\in \real^\ell$, with $\sum_j |\xi_j-\xi'_j|\le\eps$, and any $\epmtime<\epmtime_0$.
\end{lem}

\begin{proof}
  Without loss of generality, we assume $t_0=0$.

  Let $U\subset \real^\ell$ be a neighborhood of $\xi$ such that $E_{0,\mathcal F}$ has still rank $n$ when restricted to it. Then, there exists $B=\{x\colon\: \sum_j|x_j-\xi_j|\le\eps\}\subset U$ such that  $E_{0,\mathcal F}$ maps diffeomorphically a neighborhood of $B$ in $U$ onto a neighborhood of $q$. It follows, from standard properties of differential equations, that, for $\epmtime>0$ sufficiently small, the map $ E_{\epmtime,\mathcal F}$ is well defined on $B$ and that $E_{\epmtime,\mathcal F}\longrightarrow E_{0,\mathcal F}$ as $\epmtime\downarrow 0$ in the $C^1$-topology over $B$. Thus, there exists $\epmtime_1>0$ such that, for $\epmtime<\epmtime_1$, $E_{\epmtime,\mathcal F}$ has rank $n$ at every point of $B$. 

  Since the map $E_{0,\mathcal F}$ is an homeomorphism from $B$ onto a neighborhood of $q$, and $E_{\epmtime,\mathcal F}\longrightarrow E_{0,\mathcal F}$ uniformly as $\epmtime\downarrow 0$,  it follows that there exists a fixed neighborhood $V$ of $q$ and $\epmtime_2>0$ such that $V\subset E_{\epmtime,\mathcal F}(B)$, for any $\epmtime<\epmtime_2$. Then, for any $\epmtime<\min\{\epmtime_1,\epmtime_2\}$, there exists $\xi'\in B$ such that the point $q'=E_{\epmtime,\mathcal F}(\xi')$ is \eqref{cs:td}-normally reachable.
\end{proof}

We will use the following consequence of Lemma~\ref{lem:suss}. 
We remark that the result holds even if $\{f_1^{t},\ldots,f_m^{t}\}_{t\in\inter}$ satisfies the Hörmander condition only at time $t_0\in\inter$.

\begin{lem}\label{lem:dsr}
 	Let $\dsr$ be the sub-Riemannian distance induced by $\{f_1^{t_0},\ldots,f_m^{t_0}\}$, then for any $t_1\in \inter$, such that $t_1-t_0>0$ is sufficiently small, and for any $q,q'\in M$ it holds 
 	\[
 		\inf\{ \cost(\gamma)\colon\: \gamma:[t_0,t_1]\to M \text{ is \eqref{cs:td}-admissible,}\, \gamma(t_0)=q \text{ and } \gamma(t_1)=q' \}\le \dsr(q,q').
 	\]
 	In particular, $\dtd(q,q')\le \dsr(q,q')$.
\end{lem}

\begin{proof}
	Fix $\eps>0$. By Chow's theorem it is clear that $q'$ is \eqref{cs:sr}-reacheable from $q$. 
	Moreover, since there exist \eqref{cs:sr}-normally reachable points from $q'$ arbitrarily close to $q'$ (see e.g., \cite[Lemma 3.21]{Agrachev2010}), follows that $q'$ is always \eqref{cs:sr}-normally reacheable from $q$ by $\xi$ such that $\sum_j |\xi_j|\le \dsr(q,q')+\eps/2$.
	Hence, by Lemma~\ref{lem:suss}, if $\eps$ and $\eta>0$ are sufficiently small, we have that $q'$ is \eqref{cs:td}-normally reachable from $q$ at time $t_0$ by $\xi'$ such that $\sum_j |\xi'_j|\le \dsr(q,q')+ \eps$ and $T<t_1$. This clearly implies that
	\[
		\inf\{ \cost(\gamma)\colon\: \gamma \text{ is \eqref{cs:td}-admissible,}\, \gamma(t_0)=q \text{ and } \gamma(t_1)=q' \}\le \dsr(q,q')+\eps.
	\]
	Finally, the lemma follows letting $\eps\downarrow0$.
\end{proof}

We now prove the main theorem of the section.

\begin{proof}[Proof of Theorem~\ref{prop:dsrp}]
	By Lemma~\ref{lem:dsr}, we only need to prove the continuity of $\dtd$. 
	We will prove only the lower semicontinuity, sinche the upper semicontinuity follows by similar arguments.
	
	We start by proving the lower semicontinuity of $\dtd(q,\cdot)$ at $q'$.
	Consider a sequence $q_k\rightarrow q'$ and let $u_k\in \xLone([0,T_k],\real^m)$ be controls such that each one steers system \eqref{cs:td} from $q$ to $q_k$ and $\liminf_n \dd(q,q_k)=\liminf_n \|u_k\|_{\xLone}$. 
	Then, by Lemma~\ref{lem:dsr}, for any $\eps>0$ there exists a sequence of $\tilde T_k >0$ and a sequence of controls $v_k\in \xLone([T_k,\tilde T_k],\real^m)$ all steering system \eqref{cs:td} from $q_k$ to $q'$ and such that $\|v_k\|_{\xLone([T_k,\tilde T_k],\real^m)}\le \dsr(q_k,q')+\eps$. 
	Since $\dsr(q_k,q')\rightarrow 0$, this implies that 
		\[
			\dtd(q,q')\le \lim_{n\rightarrow\infty} \bigg( \|u_k\|_{\xLone([0,T_k],\real^m)}+\|v_k\|_{\xLone([T_k,\tilde T_k],\real^m])} \bigg) = \liminf_n \dtd(q,q_k) +\eps.
		\]
	Letting $\eps\downarrow 0$ proves that $\dtd(q,\cdot)$ is lower semicontinuous at $q'$.

	In order to prove the lower semicontinuity of $\dtd(\cdot, q')$ at $q$, let us define
	\[
		\varphi_\eps(p)=\inf\{ \cost(\gamma)\colon\: \gamma:[\eps,T]\subset\inter\to M \text{ is \eqref{cs:td}-admissible and } \gamma:p\rightsquigarrow q' \}.
	\]
	We claim that for any $p\in M$ it holds that $\varphi_\eps(p)\longrightarrow\dtd(p,q')$ as $\eps\downarrow0$. 
	Since it is clear that $\varphi_\eps(\cdot)\ge\dtd(\cdot,q')$, it suffices to prove that 
	\begin{equation}\label{eq:limphi}
		\lim_{\eps\downarrow 0} \varphi_\eps(p) \le \dtd(p,q')\qquad \text{for any } p\in M.
	\end{equation}
	To this aim, fix $p\in M$ and $\eta>0$ and let $\gamma:[0,T]\to M$ be such that $\cost(\gamma)\le\dtd(p,q')+\eta$.
	It is clear that $\gamma(2\eps)\rightarrow p$ as $\eps\downarrow0$, and hence that $\dtd(p,\gamma(2\eps))\rightarrow 0$ as $\eps\downarrow 0$, by the first part of the proof.
	Thus, for any $\eps>0$ sufficiently small, there exists a \eqref{cs:td}-admissible curve $\gamma_\eps:[\eps,2\eps]\to M$ such that $\gamma_\eps:p\rightsquigarrow\gamma(2\eps)$ and $\cost(\gamma_\eps)\le \dtd(p,\gamma(2\eps))+\eta$.
	By concatenating $\gamma_\eps$ with $\gamma|_{[2\eps,T]}$, we get that
	\[
		\varphi_\eps(p)\le \cost(\gamma_\eps)+\cost(\gamma)\le \dtd(p,\gamma(2\eps)) +\dtd(p,q')+2\eta.
	\]
	Letting $\eps\downarrow 0$ and then $\eta\downarrow 0$, this proves \eqref{eq:limphi} and thus the claim.

	Let now $q_k\rightarrow q$ and fix $\eta>0$. 
	By Lemma~\ref{lem:dsr} this implies that $\dtd(q_k,q)\rightarrow 0$ and that for any $\eps>0$ sufficiently small, there exists a \eqref{cs:td}-admissible curve $\gamma_\eps:[0,\eps]\to M$ such that $\gamma_\eps:q_k\rightsquigarrow q$ and $\cost(\gamma_\eps)\le\dtd(q_k,q)+\eta$.
	Hence
	\[	
		\dtd(q_k,q')\le c(\gamma_\eps)+\varphi_\eps(q) \le\dtd(q_k,q)+\varphi_\eps(q)+\eta.
	\]
	By the previous claim, letting $\eps,\eta\downarrow 0$, this implies that $\dtd(q_k,q')\le \dtd(q_k,q)+\dtd(q,q')$. 
	Since $\dtd(q_k,q)\rightarrow 0$, taking the liminf as $k\rightarrow+\infty$, this proves the lower semicontinuity of $\dtd(\cdot,q')$ at $q$, completing the proof.
\end{proof}

\begin{rmk}
    From the proof of Theorem~\ref{prop:dsrp}, it follows that hypothesis \eqref{hyp:shc} is not sharp.
    Indeed, the following is sufficient to prove the theorem.
    \begin{equation}\tag{H$_{1}$}\label{hyp:swhc}
    	\begin{split}
    		&\text{The family of smooth vector fields }\{\td t 1,\ldots, \td t m\}_{t\in \inter},\, \text{depends smoothly on } t,\, \text{}\\
    		&\text{and satisfies the strong H\"ormander condition at $t=0$ and in an open}\\
    		&\text{neighborhood of } \sup \inter.
    	\end{split}
    \end{equation}
    We will conclude this section by showing that, in our framework, it is necessary to assume the Hörmander condition on both ends of $\inter$.
    Although outside the scope of the present work, we remark that stronger assumptions on the regularity of the vector fields, i.e., that they are uniformly Lipschitz, would allow to prove Theorem~\ref{prop:dsrp} assuming only that $\{\td t 1, \ldots, \td t m \}_{t\in\inter}$ satisfies the Hörmander condition at one time $t_0\in\inter$.
\end{rmk}

The following example proves that if the family $\{\td t 1,\ldots, \td t m\}_{t\in \inter}$ satisfies the  Hörmander condition only near $t=0$, then the value function is in general not continuous. 
Through a slight modification, the same argument can also be used to prove that the same holds if the  Hörmander condition is satisfied only at a neighborhood of $\sup\inter$ or of any $t_0\in \inter$. 
	
\begin{example}\label{ex:lip}	
	\begin{figure}
	     \includegraphics[width=0.5\textwidth]{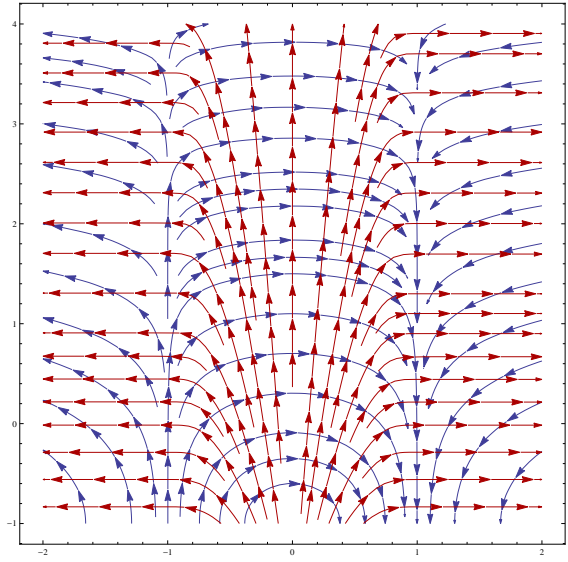}
	   \caption{The two vector fields of Example~\ref{ex:lip} with $h(x)=c\,e^{-\frac 1 {1-x^2}}$ for $x\in [-1,1]$.}
	   \label{fig:ciao}
	\end{figure}
	Let $M=(-2,2)\times(-1,+\infty)$, with coordinates $(x,y)$, and consider the vector fields 
	\[
		f(x,y)= \frac {((y+1)(1-x^2),-x)} {\sqrt{(y+1)^2(1-x^2)^2+x^2}} ,
	    \qquad g(x,y)=\frac {\big(x,h(x)(y+2)\big)}{\sqrt{x^2+h(x)^2(y+2)^2}},
	\]
	where $h:[-2,2]\to\real$ is a smooth cutoff function such that $\supp h\subset [-1,1]$, $h\ge 0$ and $h(0)=1$ (see Figure~\ref{fig:ciao}).
	Fix $0<\eps<1$, $C\ge 16$ and let $\phi,\,\psi:[0,1]\to\real$ be two smooth functions such that  
	\[
		\phi(t)=
		\begin{cases}
		    1 &\quad \text{if } 0\le t \le \eps, \\
		    0 &\quad \text{if } 2\eps\le t \le 1, \\
		\end{cases} \qquad
		\psi(t)=
		\begin{cases}
		    1 &\quad \text{if } 0\le t \le 2\eps, \\
		    C &\quad \text{if } 3\eps\le t \le 1, \\
		\end{cases}
	\]
	and such that $\phi$ is nonincreasing while $\psi$ is nondecreasing.
	Finally, consider the time-dependent system on $M$ specified by the vector fields $f^t(x,y)=\phi(t)f(x,y)$ and $g^t(x,y)=\psi(t)g(x,y)$, $t\in[0,1]$. 
	We will show that $\{f^t,g^t\}$ satisfies the Hörmander condition for $t\in[0,\eps]$, but that the value function associated with the family $\{f^t,g^t\}_{t\in[0,1]}$ is not lower semicontinuous.

	We start by showing that $f(p)$ and $g(p)$ are transversal for any $p=(x,y)\in M$, implying the Hörmander condition for $\{f^t,g^t\}$, $t\in[0,\eps]$.
	If $x\in (-2,-1]\cup[1,2)$, then, by definition of $h$,  $g(p)=(1,0)$ is clearly transversal to $f(p)$.
	On the other hand, if $x\in (-1,1)\setminus \{0\}$ and $g(p)$ is parallel to $f(p)$, a simple computation shows that $h(x)<0$, which is a contradiction. 
	Finally, for $x=0$, it is clear that $g(p)=(0,y+2)$ and $f(p)=(y+1,0)$ are never parallel.
	We remark that this implies also that the value function $\rho_{\eps}$, induced by controls defined on $[0,\eps]$, is a distance equivalent to the Euclidean one.
	In particular, $|p_1-p_2|\le2\rho_\eps(p_1,p_2)$ for any $p_1,p_2\in M$.

	Fix now $q'=(1,0)$.
	The set of points from which $q'$ is reachable using only $f$ is exactly $\mathcal O_{q'}=\{ (1,y)\colon\: y>-1\}$. 
	Let then $q_0\in (-1,0)\times\{0\}$ be such that $\rho_\eps(q_0,(-1,0))\le \frac 1 4 \min_{p\in \mathcal O_{q'}} \rho_\eps(q_0,p)$. 
	In order to show that $\rho_1(q_0,\cdot)$ is not lower semicontinuous at $q'$, consider any sequence $\{q_n\}_{n\in\mathbb N}\subset(1/2,1)\times\{0\}$ such that $q_n\longrightarrow q'$.
	By continuity of $\rho_\eps$ and the fact that $-q_n\longrightarrow (-1,0)$, we can always assume that, up to subsequences, $\rho_\eps(q_0,-q_n)\le \frac 1 2 \min_{p\in \mathcal O_{q'}} \rho_\eps(q_0,p)$.
	
	Since $g^t\equiv 0$ for $t\ge2\eps$, 
	if $u\in \xLone([0,1],\mathbb R^2)$ is a control steering the system from $q_0$ to $q'$, the control $u|_{[0,2\eps]}$ steers the system from $q_0$ to some $p\in\mathcal O_{q'}$. Exploiting the fact that $\rho_{2\eps}\ge\rho_\eps$ by monotonicity of $\psi$, this implies that
	\begin{equation}\label{eq:rho1}
		\rho_1(q_0,q')\ge \min_{p\in \mathcal O_{q'} } \rho_\eps(q_0,p) \ge 2 \rho_\eps(q_0,-q_n).
	\end{equation}

	Let now $u\in \xLone([0,1],\mathbb R^m)$ be the control constructed as follows.
	From time $0$ to $\eps$, $u|_{[0,\eps]}$ is the minimizer of $\rho_\eps$ steering the system from $q_0$ to $-q_n$.
	Then, $u|_{(\eps,3\eps)}\equiv 0$ and, after this, the control acts only on $f^t$ for time $t\in[3\eps,1]$, steering the system from $-q_n$ to $q_n$.
	Hence,
	\begin{equation}\label{eq:rho2}
		|q_n-(-q_n)|=\left| \int_{3\eps}^1 u(t)f^t(x(t),y(t))\,dt \right| =C \int_{3\eps}^1 |u(t)|\,dt.
	\end{equation}
	Since $|q_n-(-q_n)|<2$, $C\ge 16/|q_0-q'|\ge 8/\rho_\eps(q_0,q')$, and by \eqref{eq:rho2}, it holds that
	\[
		\rho_1(q_0,q_n)\le \int_0^1 |u(t)|\,dt = \rho_\eps(q_0,-q_n) + \frac 1 C {|q_n-(-q_n)|} \le \frac 3 4 \rho_1(q_0,q').
	\]
	Taking the $\liminf$ as $n\rightarrow \infty$ shows that $\rho_1(q_0,\cdot)$ is not l.s.c. at $q'$.
\end{example}

\subsection{Estimates on reachable sets}\label{subsec:etd}
	
	In this section, we concentrate on a particular class of time-dependent systems. Namely, 
	let $\{f_1,\ldots,f_m\}$ be a bracket-generating family of smooth vector fields, $f_0$ be a smooth vector field, and 
	consider the time-dependent system
	\begin{equation}\label{cs:td2}
		\dot q=\sum_{i=1}^m u_i\, \td t i,\qquad \td t i = (e^{-tf_0})_*f_i(q),\quad q\in M,\quad u=(u_1,\ldots,u_m)\in\real^m.
	\end{equation}
	Here, $(e^{-tf_0})_*$ is the push-forward operator associated with the flow of $f_0$.  

	As we will see in the next section, this class of systems  arises naturally when dealing with control systems that are affine with respect to the control. 
	Observe, in particular, that from the bracket-generating property of $\vett$ it follows immediately that the time-dependent family $\{(e^{-t\drift})_*f_1,\ldots, (e^{-t\drift})_*f_m\}$ satisfies the strong H\"ormander condition, as per Definition~\ref{def:hor}.
	
	Before proceeding with the estimates of the reachable sets, we need to define a suitable approximation of system \eqref{cs:td2}.
	Namely,	fix a system of privileged coordinates (in the sub-Riemannian sense) at $q$ for $\vett$. Assume that $\drift(q)\neq0$, and let $s\in\{1,\ldots,r\}$ be such that $\ord_q(\drift)=-s$. In this case, there exist,  in coordinates, an homogeneous vector field $\drift^{-s}$, of weighted degree $-s$, and a vector field $\drift^{>-s}$, of weighted degree $\ge -s+1$, such that
	\begin{equation}\label{eq:f0}
		z_*\drift=  \drift^{-s}+\drift^{>-s}.
	\end{equation}
	In particular, it holds that $\drift^{-s}\not\equiv0$ near $z(q)=0$.
	
	\begin{rmk}\label{rmk:order}
		The fact that $\ord_q(\drift)=-s$ is not equivalent, in general, to   $\drift\in\distr^s$ near $q$. 
		In particular, if the growth vector is non-constant around $q$, from $\ord_q(\drift)\ge-s$ it does not follow that $\drift\in\distr^{s}$. For example, consider the sub-Riemannian control system on $\real^2$ with (privileged) coordinates $(x,y)$, defined by the vector fields $\partial_x$ and $x\partial_y$,  called the Grushin plane. Outside $\{x=0\}$, the non-holonomic degree of these vector fields is 1, while, on $\{x=0\}$, we need one bracket to generate the $y$ direction, and thus it is $2$. 
		Hence, if $\bar y\neq 0$ the vector field $y\partial_y$ is never in $\distr$ near $(0,\bar y)$, but $\ord_{(0,\bar y)}(y\partial_y)=\ord_{(0,\bar y)}(y)+\ord_{(0,\bar y)}(\partial_y)=0$.

		However, due to the properties \eqref{eq:ord} and the fact that $\distr^s$ is a module, the converse is always true. Namely, if $\drift\in\distr^s$, then $\ord_q(\drift)\ge-s$.
		
	\end{rmk}

        For any smooth vector field $f$, let $\adj 1 \drift f=[\drift,f]$ and  $\adj \ell \drift f = [\drift,\adj {\ell-1} \drift f]$, for any $\ell\in\nat$.
	We recall (see for example \cite{Hermes1991}) that the following Taylor expansion holds
	\begin{equation}\label{eq:push}
		(e^{-tf_0})_*f \sim \sum_{\ell=0}^\infty \frac{t^\ell}{\ell!}\, \adj \ell {\drift} f.
	\end{equation}
	Since $\ord_q(f_j)\ge -1$, by \eqref{eq:ord} we have that $\ord_q(\adj \ell {\drift} f_j)\ge-\ell s-1$. Then, using the decomposition \eqref{eq:f0}, for any $\ell\ge0$, there exists, in coordinates, an homogeneous vector field $\homvf \ell j$ of weighted degree $-\ell s$, and a remainder $r^\ell$ of order $\ge-\ell s-1$, such that
	\begin{equation}\label{eq:fhatell}
		z_* \big[ \adj \ell {\drift} {f_j} \big]=\homvf \ell j + r^\ell.
	\end{equation}
	
	\begin{defn}
  	The \em \app \em at $q$ of $\td t j$, $1\le j \le m$, associated with the privileged coordinates $z$, is the vector field with coordinate representation
	\begin{equation}\label{eq:alpha}
		\ntd t j = \sum_{\ell=0}^\varrho \frac{t^\ell}{\ell!}\,\homvf \ell j,
	\end{equation}
	where $\varrho=\lfloor \nicefrac{r-1}{s} \rfloor$ and $r$ is the non-holonomic degree of $\{f_1,\ldots,f_m\}$ at $q$.
	The approximated time-dependent control system is then defined as
	\begin{equation}
	\tag{ATD}\label{cs:ntd2}
	\dot q = \sum_{j=1}^m u_j(t) \ntd t j(q).\\
	\end{equation}
	If a system, in some system of privileged coordinates, coincides with its \app, we will say that it is \appsys. 
	\end{defn}
	
	The \app encodes the idea that the time $t$ is of weight $s=-\ord_q(\drift)$. 
	This is a consequence of the fact that, due to the expansion \eqref{eq:push}, $t$ allows to build brackets of $f_0$ with the $f_j$s.
	In this sense, the \app is a generalization of the nilpotent approximation.
	
	We are now ready to state the main theorem of this section.
	\begin{thm}\label{thm:ballbox}
	Let $\vett$ satisfy the H\"ormander condition, let $z=(z_1,\ldots,z_n)$ be a system of privileged coordinates at $q\in M$ for $\vett$. 
	Then there exist $C,T,\eps_0>0$ such that, for any $\eps<\eps_0$ and any $q'\in\btdt {q} \eps T$, setting $s=-\ord_q(\drift)$ it holds
	\begin{gather}\label{eq:bb}
          |z_i(q')|\le 
          C\bigg(  \eps^{w_i}+\eps T^{\frac{w_i}{s}} \bigg)\quad \text{if }w_i\le s, \\
          \label{eq:bb2}
          |z_i(q')|\le C\eps\bigg(  \eps+T^{\frac{1}{s}} \bigg)^{w_i-1}\quad \text{if }w_i> s.
	\end{gather}

	Moreover, if the system is \appsys, then it holds the stronger estimate 
	\begin{equation}\label{eq:bbn}
		|z_i(q')|\le 
			C  \eps^{w_i} \quad \text{if }w_i\le s. 
	\end{equation}
	\end{thm}

	To prove this theorem we need the following proposition, estimating the difference between \eqref{cs:td2} and \eqref{cs:ntd2}.  
	
	\begin{prop}\label{prop:ngenappr}
	Let $\vett$ satisfy the H\"ormander condition,
	and let  $z=(z_1,\ldots,z_n)$ be a system of privileged coordinates at $q\in M$ for $\vett$. 
	For  $T>0$ and $u\in \xLone([0,T];\real^m)$, let $\gamma(\cdot)$ and $\hat \gamma(\cdot)$ be the trajectories associated with $u$ in \eqref{cs:td2} and \eqref{cs:ntd2}, respectively, and such that $\gamma(0)=\hat\gamma(0)=q$. 
	Then there exist $C,\eps_0,T_0>0$, independent of $u$, such that, if $t<T_0$ and $\int_0^t|u|\,ds= \eps<\eps_0$, and setting $s=-\ord_q(\drift)$ it holds 		
	\begin{equation}\label{eq:nil}
		|z_i(\gamma(t))-z_i(\hat \gamma (t))|\le C \eps \big( \eps+t^{ \frac {1} s} \big)^{w_i},\qquad i=1,\ldots,n.
	\end{equation}
\end{prop}

\begin{rmk}
	This proposition generalizes Proposition~\ref{prop:nilsr}. In fact, in the sub-Riemannian case, since $f_0\equiv 0$, if, for $t>0$, the curve $\gamma$ is associated to  $u\in \xLone([0,t],\real^m)$, it is associated also to $u_\tau(\cdot)=\frac \tau t u(\frac \tau t \cdot)$, for any $\tau>0$. 
	Thus, since $\int_0^\tau |u_\tau|\,ds=\int_0^t |u|\,ds=\eps$, \eqref{eq:nil} reduces to
	\[
		|z_i(\gamma(t))-z_i(\hat\gamma(t))|\le \lim_{\tau\downarrow 0}C(\eps^{w_i+1}+\tau\eps^{r}) = C \eps^{w_i+1}.
	\]
	Finally, assuming that $u$ satisfies  the hypotheses of Proposition~\ref{prop:nilsr}, i.e., that  $|u|=1$, we get  $t=\eps$. 
\end{rmk}

\begin{proof}
	Let $z(\gamma(\cdot))=x(\cdot)$, $z(\hat \gamma(\cdot))=y(\cdot)$,  and $\|z\|=\sum_{\ell=1}^n |z_\ell|^{1/w_i}$. We mimic the proof of Proposition~7.29 in \cite{Bellaiche1996}. 
	The first step is to prove that there exists a constant $C>0$ such that $\|x(t)\|,\,\|y(t)\|\le C \eps$ for  $t$ and $\eps=\int_0^t |u|\,ds$ small enough. We prove this for $\|x(t)\|$, the same argument works also for $\|y(t)\|$.
	
	In $z$ coordinates, the equation of the control system \eqref{cs:td2} is,	\begin{equation*}
		\dot x_i(t) = \sum_{j=1}^m u_j(t) (z_i)_*\,\td t j (\gamma(t)),\qquad i=1,\ldots, n.
	\end{equation*}
	Due to the fact that $z_*\td t j = z_* f_j + \mathcal O(t)$, uniformly in a neighborhood of $q$
	that $\ord_q(z_i)=w_i$ and that $\ord_q(f_j)\ge-1$, we have that there exist $T_0$ and $C>0$ such that $|(z_i)_* \td t j (q)|\le \frac C 2 |(z_i)_* f_j(q)|\le C \| x(t)\|^{w_i-1}$, for any $t<T_0$. Thus we get
	\begin{equation}\label{eq:rise}
		|\dot x_i(t)| \le C \sum_{j=1}^m |u_j(t)| \|x_i(t)\|^{w_j-1}.
	\end{equation}
	As in the proof for the sub-Riemannian case, choosing $N$ sufficiently large, so that all $N/w_i$ are even integers, and setting $|||z|||=(\sum_{\ell=1}^n |z_\ell|^{N/w_i})^{\frac 1 N}$ we get a norm equivalent to $\|z\|$. Deriving with respect to time and using \eqref{eq:rise} we get $\frac{d}{dt}||| x(t) |||\le C \sum_{j=1}^n|u_j(t)| $. Finally, by integration, equivalence of the norms, and the fact that $x(0)=z(q)=0$, we conclude that $\|x(t)\|\le C\eps$.

	Now we move to proving \eqref{eq:nil}.
	By construction of \eqref{cs:ntd2} and the Taylor expansion of $\td t j$, for any $\ell\le \varrho=\lfloor \nicefrac{r-1}{s}\rfloor$, there exist homogeneous polynomials $h_{ji}^\ell$ of order $w_i-\ell\sM-1$ and remainders $r_{ji}^\ell$ of order larger than or equal to $w_i-\ell\sM$, such that we can write
	\begin{gather*}
		(z_i)_*\td t j=\sum_{\ell=0}^{\varrho} \frac {t^\ell}{\ell!} \big(h_{ji}^\ell + r_{ji}^\ell \big)+ \bigo(t^{\varrho+1}),\\
		(z_i)_*\ntd t j=\sum_{\ell=0}^{\varrho} \frac {t^\ell}{\ell!}  h_{ji}^\ell .
	\end{gather*}
	Here, the $\bigo$ is intended as $t\downarrow 0$ and is uniform in a compact neighborhood of the origin.
	Then, 
	\[
	\begin{split}
		\dot x_i(t)- &\dot y_i(t) = \sum_{j=1}^m u_j(t) \bigg(  \sum_{\ell=0}^{\varrho} \frac {t^\ell}{\ell!} \big(h_{ji}^\ell(x)- h_{ji}^\ell(y) + r_{ji}^\ell(x) \big)+  \bigo({t^{\varrho+1}})\bigg) \\
			 &\quad= \sum_{j=1}^m u_j(t) \bigg(  \sum_{\ell=0}^{\varrho} \frac {t^\ell}{\ell!}  \bigg(\sum_{w_k<w_i-\ell\sM} \big(x_k(t)-y_k(t) \big)Q_{jik}^\ell(x,y)+ r_{ji}^\ell(x) \bigg)+\bigo({t^{\varrho+1}}) \bigg),
	\end{split}
	\]
	where $Q_{jik}^	\ell$ are homogeneous polynomial in $x$ and $y$, of order $w_i-w_k-\ell\sM-1$. We observe that, if $w_i-w_k-\ell\sM-1<0$, then $Q^\ell_{jik}\equiv0$. Thus, for sufficiently small $\|x\|$ and $\|y\|$, we have 
	\begin{gather*}
		|Q_{jik}^\ell(x,y)|\le C\big( \|x\|^{(w_i-w_k-\ell\sM-1)^+} + \|y\|^{(w_i-w_k-\ell\sM-1)^+}   \big), \quad
		|r_{ji}^\ell(x)|\le C\|x\|^{(w_i-\ell\sM)^+}.
	\end{gather*}
	Here we let $(\xi)^+=\max\{\xi,0\}$, for any $\xi\in\real$.
	Using the inequalities of the first step, taking $t<T$ sufficiently small, and enlarging the constant $C$, we get
	\[
	\begin{split}
		|\dot x_i(t)-&\dot y_i(t)| \\
			&\le C |u(t)| \bigg(  \sum_{\ell=0}^{\varrho} \frac {t^\ell}{\ell!} \bigg(\sum_{w_k<w_i-\ell\sM} \big|x_k(t)-y_k(t) \big| \eps^{w_i-w_k-\ell\sM-1}+ \eps^{(w_i-\ell\sM)^+} \bigg)+  {t^{\varrho+1}} \bigg)\\
			&\le C|u(t)| \bigg(\sum_{\ell=0}^{\varrho} t^\ell\bigg( \sum_{w_h<w_i} \big|x_h(t)-y_h(t) \big| \eps^{w_i-w_h-1}+ \eps^{(w_i-\ell\sM)^+}\bigg) +{t^{\varrho+1}}  \bigg).
	\end{split}
	\]
	In the last inequality we applied the change of variable $w_k\mapsto w_h-\ell\sM$ in each of the sums.
			
	We can integrate the system by induction, since it is in triangular form.  For $w_i=1$, since $(w_i-\ell s)^+=0$ for any $\ell\ge 1$, the inequality reduces to 
	\[
	|\dot x_i(t)-\dot y_i(t)|\le C |u(t)|\left(\sum_{\ell=0}^{\varrho} t^\ell\eps^{(w_i-\ell\sM)^+}+t^{\varrho+1}\right)\le C |u(t)|(\eps^{w_i} + t).
	\]
	Here we enlarged the constant $C$. Thus, integrating the previous inequality, we get 
	$
		|x_i(t)-y_i(t)| \le C \eps (\eps^{w_i} + t)\le C \eps (\eps + t^{\frac 1 s})^{w_i}.
	$
	
	Let, then, $w_i>1$ and assume that $ |x_h(t)-y_h(t)|\le C  \eps (\eps+t^{\frac{1}s} )^{w_h}$ for any $w_h<w_i$. 
	To complete the proof it suffices to show that $|\dot x_i (t)-\dot y_i(t)|\le C|u(t)|(\eps+t^{\frac{1} s})^{w_i}$, since \eqref{eq:nil} will follow, as above, by integration.
	Thus, we have, enlarging again the constant $C$ and taking $t$ sufficiently small,
	\begin{equation}\label{eq:ciao}
		\begin{split}
			|\dot x_i & (t)-\dot y_i(t)| \\
				&\le C |u(t)| \bigg( \sum_{\ell=0}^{\varrho} t^\ell\left(\sum_{w_h<w_i} \left(\eps+t^{\frac{1}s} \right)^{w_h} \eps^{w_i-w_h} + \eps^{(w_i-\ell\sM)^+} \right) + t^{\varrho+1} \bigg)\\
				&\le C |u(t)| \bigg( \sum_{\ell=0}^{\varrho} t^\ell\left(\sum_{w_h<w_i} t^{\frac{w_h}s} \eps^{w_i-w_h} + \eps^{(w_i-\ell\sM)^+} \right) + t^{\varrho+1} \bigg).
		\end{split}
	\end{equation}
	 If $t\le \eps^s$, from \eqref{eq:ciao} it is clear that $|\dot x_i (t)-\dot y_i(t)| \le C|u(t)|\eps^{w_i}$. Here we used the fact that $\varrho+1\ge w_i/s$. On the other hand, if $\eps< t^{\nicefrac{1}{s}}$, it holds
	\[
	|\dot x_i (t)-\dot y_i(t)|\le C|u(t)| \left(\sum_{\ell=0}^{\varrho} \left( \sum_{w_h<w_i} t^{\frac{w_h}s+\ell+\frac{w_i-w_h}{s}} +t^{\ell+\frac{w_i-\ell s} s}\right) +t^{\varrho+1} \right)\le C |u(t)| t^{\frac{w_i} s}.
	\]
	Putting together these two estimates, we get that $|\dot x_i (t)-\dot y_i(t)|\le C|u(t)|(\eps^{w_i}+t^{\frac{w_i} s})\le C|u(t)|(\eps+t^{\frac{1} s})^{w_i}$, completing the proof of the proposition.
\end{proof}

\begin{proof}[Proof of Theorem~\ref{thm:ballbox}]
	We start by claiming that \eqref{eq:bbn} implies \eqref{eq:bb}. In fact, if $\gamma:q\rightsquigarrow q'$ is the trajectory associated in \eqref{cs:td2} to a control $u\in \lcont$, and $\hat\gamma$ is the trajectory associated with the same control in the \app \eqref{cs:ntd2}, with $\hat\gamma(0)=q$, it holds
	\[
		|z_i(q')|\le|z_i(\hat\gamma(T))|+|z_i(\hat\gamma(T))-z_i(\gamma(T))|.
	\]
	Thus, by Proposition~\ref{prop:ngenappr}, the claim is proved. 
	
	Hence, from now on we assume our system to be in the form \eqref{cs:ntd2}.
	Let us define, for $1\le j \le n$ and $0\le\alpha\le r$,  the vector fields $\varphi_j^\alpha$ as
	\[
		\varphi_j^\alpha=\sum_{\ell=0}^{\alpha} \frac {t^\ell}{\ell!} \homvf \ell j,
	\]
	where $\homvf \ell j$ are defined in \eqref{eq:fhatell}. 
	We do not explicitly denote the dependence on time, to lighten the notation. Observe that, if $\alpha=\varrho$, then, by \eqref{eq:alpha}, $\varphi_j^\alpha=\ntd t j$.
	
	We claim that, letting $x^{(\alpha)}(\cdot)$ be the trajectory associated with a control $u\in\lcont$ in system \eqref{cs:td} with $\{\varphi_1^\alpha,\ldots,\varphi_m^\alpha\}$ as vector fields, then, for some constant $C>0$ and any $i\in\{1,\ldots,n\}$ and $\alpha\ge 1$, it holds
	\begin{equation}\label{eq:al}
		|x^{(\alpha)}_i(T)-x^{(\alpha-1)}_i(T)| \le
		\begin{cases}
			0 &\quad \text{if } w_i\le\alpha s,\\
		 	C \eps(\eps+ T^{\frac{1}{s}})^{w_i-1}&\quad \text{if } w_i>\alpha s.
		 \end{cases}
	\end{equation} 	
	In fact, due to the homogeneity of the $\homvf \ell j$, proceeding as in the proof of Proposition~\ref{prop:ngenappr}, we get that for $w_i\le\alpha s$ it holds
	\begin{equation*}
		|\dot x^{(\alpha)}_i(t)-\dot x^{(\alpha-1)}_i(t)| \le C |u(t)| \sum_{\ell=0}^{\alpha-1} t^\ell \sum_{w_h<w_i} | x^{(\alpha)}_h(t)- x^{(\alpha-1)}_h(t)| \eps^{w_i-w_h-1}.
	\end{equation*}
	By induction on $1\le w_i\le \alpha s$, this proves the first part of the claim. On the other hand, if $w_i > \alpha s$, it holds
	\begin{equation*}
		|\dot x^{(\alpha)}_i(t)-\dot x^{(\alpha-1)}_i(t)|
		\le C |u(t)|\bigg( \sum_{\ell=0}^{\alpha-1} t^\ell \sum_{w_h<w_i} | x^{(\alpha)}_h(t)- x^{(\alpha-1)}_h(t)| \eps^{w_i-w_h-1} + t^{\alpha} \eps^{w_i-\alpha\sM-1}\bigg).
	\end{equation*}
	Then, again by induction over $w_i$, we get that $|x^{(\alpha)}_i(T)-x^{(\alpha-1)}_i(T)|\le C T^{\alpha}\eps^{w_i-\alpha s}$. Finally, the claim follows considering the cases $T\le \eps^s$ and $T>\eps^s$.
		
	Due to the fact that $\varphi_{j}^0=\widehat f_j$, by Theorem~\ref{thm:srballbox} it holds $|x_i^{(0)}(T)|\le C\eps^{w_i}$. 
	Thus, applying \eqref{eq:al} and enlarging the constant $C$, we get 
	\[
		|z_i(q')|=|x_i^{(r)}(T)|  \le \sum_{\ell=1}^{r} \left|x^{(\ell)}_i(T)-x^{(\ell-1)}_i(T)\right| +\left|x_i^{(0)}(T)\right| 
		\le 
		\begin{cases}
			C\eps^{w_i} & \text{if } w_i\le s,\\
			C   \eps(\eps+ T^{\frac{1}{s}})^{w_i-1}& \text{if } w_i> s.\\
		\end{cases}
	\]
	This proves \eqref{eq:bb2} and  \eqref{eq:bbn}, completing the proof of the theorem.
	\end{proof}	
	
	We end this section by showing that the estimate \eqref{eq:bb2} is sharp, at least in some directions.  
	Indeed, for a system which is \appsys $\text{ }$ at $q$ in some privileged coordinates $z$, and satisfies the hypotheses of Theorem~\ref{thm:ballbox}, it holds that $z_*(\adj k {\drift} f_j)$ is an homogeneous vector field of weighted degree $-sk-1$. Thus, since $\eps t^k\le \eps(\eps+t^{\frac{1}{s}})^{sk}$, the following proposition shows that \eqref{eq:bb2} is sharp in this direction. The proof is an adaption of an argument from \cite{Coron2009}. 
	
	\begin{prop}\label{prop:inner}
	Let $\vett$ satisfy the H\"ormander condition. Let, moreover
		$q\in M$, $i\in\{1,\ldots,m\}$ and $k\ge 0$. Then, for any coordinate system $y$ at $q$, there exist $T,\eps_0>0$ such that, for any $\eps<\eps_0$ and $t<T$ there exists
		 a \eqref{cs:td}-admissible curve $\gamma:[0,t]\to M$, with $\cost(\gamma)\le\eps$, and such that
		 \[
		 	y(\gamma(t))= \eps t^{k} dy\big(\adj k {\drift} f_j(q)\big)+\bigo(\eps t^{k+1}) \quad \text{ as } \eps t\rightarrow 0.
		 \]
	\end{prop}
		
	\begin{proof}
	Let $t,\eta>0$ be fixed, and define $u\in\lcont$ as $u_i(\tau) \equiv \eta$, $u_j(\tau)\equiv 0$  for ${j\neq i}$, $\tau\in[0,t]$.	
	Then, fix any $\Phi\in C^{k}([0,1])$ such that $\Phi^{(i)}(0)=\Phi^{(i)}(1)=0$, for ${0\le i <k}$. Thus, by integrating by parts and the fact that $\frac d {dt} (e^{-tf_0})_*\,g=(e^{-tf_0})_* \big(\text{ad}( {f_0}) g\big)$, we get
	\[
		\int_0^t \Phi^{(k)}(\nicefrac{\tau}{t}) (e^{-\tau f_0})_* f_i(q)\, d\tau= t^k \int_0^t    \Phi(\nicefrac{\tau}{t}) (e^{-\tau f_0})_*\left( \adj k {f_0} f_i\right)(q)\,d\tau,
	\]
	for any $t$ and $q$.
	This implies that the flows generated by $\Phi^{(k)}(\nicefrac{\tau}{t}) (e^{-\tau f_0})_* f_i$ and $t^k \Phi(\nicefrac{\tau}{t}) (e^{-\tau f_0})_*\left( \adj k {f_0} f_i\right)$ coincide.
	Using the series expansions of the chronological exponential and $(e^{-tf_0})_*$, see \cite[Section 2.4]{Sachkov2004}, there holds, then,
	\[
		\begin{split}
			\crex_0^t \sum_{j=1}^m   \Phi^{(k)}(\nicefrac{\tau}{t}) u_j(\tau)  (e^{-\tau f_0})_*f_j \, d\tau  
			&= \crex_0^t  \eta  \Phi^{(k)}(\nicefrac{\tau}{t}) (e^{-\tau f_0})_*f_i \, d\tau\\
			&= \crex _0^t \eta t^k \,\Phi(\nicefrac{\tau}{t}) (e^{-\tau f_0})_*\left( \adj k {f_0} f_i\right)\,d\tau \\
			&= \crex _0^1 \eta t^{k+1}\, \Phi(s)  (e^{-ts f_0})_*\left( \adj k {f_0} f_i\right)\,ds  \\
			&= \crex _0^1 \eta t^{k+1} \, \Phi(s)  \left( \adj k {f_0} f_ i + \bigo(t)\right)\,ds  \\
			&=\text{Id} + \eta t^{k+1} \adj k {f_0} f_ i + \bigo(\eta t^{k+2}) 
		\end{split}
	\]
	Finally, considering any coordinate system and letting $\eps=\eta t$, this completes the proof.
	\end{proof}	
	
	\section{Control-affine systems}\label{sec:drift}
	
	In this section we apply the results of Section~\ref{sec:td} to control-affine systems. Let $\{f_1,\ldots,f_m\}$ be a bracket-generating family of vector fields, $\drift$ be a smooth vector field, called the \em drift\em, and consider the control-affine system
	\begin{equation}\label{cs:drift}\tag{D}
		\dot q=\drift(q)+\sum_{i=1}^m u_i\, f_i(q),\qquad q\in M,\quad u=(u_1,\ldots,u_m)\in\real^m.
	\end{equation}
	The assumption on $\{f_1,\ldots,f_m\}$ to be bracket-generating, is called \em strong H\"ormander condition \em for \eqref{cs:drift}.

	An absolutely continuous curve $\gamma:[0,T]\to M$ is \eqref{cs:drift}-admissible if $\dot\gamma(t)=\drift(\gamma(t))+f_{u(t)}(\gamma(t))$ for some control $u\in\lcont$. Its cost is defined as
	\[
		\costd(\gamma)=\min \|u\|_{\lcont},
	\]
	where the minimum is taken over all controls $u$ such that $\gamma$ is associated with $u$.	
	Then, we define the two value functions we are interested in as
	\begin{gather*}
		\ddt(q,q')=\inf\{ \cost(\gamma)\colon\: \gamma:[0,T']\to M \text{ is \eqref{cs:drift}-admissible, } \gamma:q\rightsquigarrow q',\,T'\le T \},\\
		\dd(q,q')=\inf\{ \cost(\gamma)\colon\: \gamma \text{ \eqref{cs:drift}-admissible and } \gamma:q\rightsquigarrow q' \}.
	\end{gather*}
	It is clear that $\ddt(q,q')\searrow\dd(q,q')$ as $T\rightarrow+\infty$, for any $q,q'\in M$. Moreover, we observe that, $\ddt(q,e^{tf_0}q)=0$ for any $0\le t\le T$.
	Finally, the reachable sets with respect to these value functions, from any $q\in M$ and for $\eps,T > 0$, are
	\[
		\bdt {q} \eps =\{q'\in M\colon\: \ddt(q,q')<\eps\},\quad \bd {q} \eps=\{q'\in M\colon\: \dd(q,q')<\eps\}.
	\]

\subsection{Connection with time-dependent systems}
	Applying the variations formula (see \cite{Sachkov2004}), system \eqref{cs:drift} can be written as a composition of a time-dependent system in the form \eqref{cs:td2} and of a translation along the drift. Namely, for any $u\in\lcont$, it holds
	\begin{equation}\label{eq:split}
		\crex_0^T \bigg( \drift+\sum_{i=1}^m u_i(t)\, f_i \bigg)\,dt  = e^{T\drift}\circ \crex_0^T  \sum_{i=1}^m u_i(t)\, (e^{-t\drift})_* f_i\,dt.
	\end{equation}
	We call \em time-dependent system associated with \eqref{cs:drift} \em the following control system,
	\begin{equation}\label{cs:ad}
		\dot q=\sum_{i=1}^m u_i\, (e^{-tf_0})_*f_i(q),\qquad q\in M,\quad u=(u_1,\ldots,u_m)\in\real^m.
	\end{equation}
	Observe that, since diffeomorphisms preserve linear independence, the strong H\"ormander condition for \eqref{cs:drift}, implies that $\{(e^{-t\drift})_*f_1,\ldots,(e^{-t\drift})_*f_m\}_{t\in[0,+\infty)}$ satisfies the strong H\"ormander condition, defined in Definition~\ref{def:hor}.
		
	Exploiting these facts, we can prove the following.
	
	\begin{prop} \label{prop:dcont}
		Assume that \eqref{cs:drift} satisfies the strong H\"ormander condition. Then, 
		for any $T>0$, the functions $\ddt,\dd:M\times M\to [0,+\infty)$ are continuous. Moreover, letting $\dsr$ be the sub-Riemannian distance induced by $\vett$, for any $q,q'\in M$ it holds 
		\[ 
		\ddt(q,q')\le\min_{0\le t \le T}\dsr(e^{t\drift}q,q'),\qquad \dd(q,q')\le\min_{t\ge 0}\dsr(e^{t\drift}q,q').
		\]
	\end{prop}
	
	\begin{proof}
		The continuity of the two functions, and the fact that $\ddt(q,q'),\dd(q,q')\le\dsr(q,q')$, for any $q,q'\in M$, follows from the same arguments used in Theorem~\ref{prop:dsrp}, adapting Lemmata~\ref{lem:suss} and \ref{lem:dsr} to the system \eqref{cs:drift}. In particular, one has to consider $(\epmtime,\mathcal F, \xi)\mapsto e^{\epmtime\drift}\circ E_{\epmtime,\mathcal F}(\xi)$ instead of $(\epmtime,\mathcal F, \xi)\mapsto E_{\epmtime,\mathcal F}(\xi)$.

		To prove the second part of the statement, we let, for any $t\in[0,T)$,
		\[
			\varphi_t(p)=\inf\{ \cost(\gamma)\colon\: \gamma:[t,T']\to M \text{ is \eqref{cs:drift}-admissible, } \gamma:p\rightsquigarrow q',\,T'\le T \}.
		\]
		It is clear that, as above, it holds $\varphi_t(p)\le\dsr(p,q')$.
		Moreover, we observe that $\ddt(q,e^{t\drift}q)=0$ for any $0\le t < T$, and hence that
		for any such $t$ it holds
		\[
			\ddt(q,q')\le \varphi_t(e^{t\drift}q)\le \dsr(e^{t\drift}q,q').
		\]
		Taking the minimum for $0\le t < T$, proves the inequality regarding $\ddt$. To complete the proof it suffices to observe that $\dd(q,q')\le \ddt (q,q')$ for any $T>0$.
	\end{proof}
	
	We point out that in system \eqref{cs:drift}, as in time-dependent systems, the existence of minimizers is not assured. Moreover, this lack of minimizers is possible even if they exist for the associated time-dependent system, as the following example points out.
	
	\begin{example}
		Consider the following vector fields on $\real^3$, with coordinates $(x,y,z)$,
		\[
			f_1(x,y,z)=\partial_x,\qquad f_2(x,y,z)=\partial_y+x\partial_z.
		\]
		Since $[f_1,f_2]=\partial_z$, $\{f_1,f_2\}$ is a bracket-generating family of vector fields. 
		The sub-Riemannian control system associated to $\{f_1,f_2\}$ on $\real^3$ corresponds to the Heisenberg group.
		
		Let now $f_0=\partial_z$ be the drift.
		Since $[f_1,\partial_z]=[f_2,\partial_z]=0$ it holds that $(e^{-t\drift})_*f_1=f_1$ and $(e^{-t\drift})_*f_2=f_2$. Hence, the associated time-dependent  system is actually not time-dependent. 
		Thus, by \eqref{eq:split}, a curve $\gamma:[0,T]\to \real^3$ is \eqref{cs:sr}-admissible for $\{f_1,f_2\}$ if and only if $\tilde\gamma(\cdot)=e^{\cdot f_0}\circ\gamma(\cdot)$ is \eqref{cs:drift}-admissible. As a consequence of this, for any $q\in\real^3$ and any $\eps>0$, 
		\[
			\bd {q} \eps= \bigcup_{t\ge 0} e^{t f_0}\circ \bsr {q} \eps.
		\]
		
		As pointed out in Section~\ref{sec:sr}, minimizers for the sub-Riemannian system exist between any pair of points in $\bsr q \eps$, if $\eps$ is sufficiently small.
		Let us show that, for any point in $ \bd q \eps$ with positive $z$ coordinate, we have an explicit minimizer, while for the others there exists no minimizer. 
		Without loss of generality we can consider $q=0$.
		Then, since $e^{t' f_0}(x',y',z')=(x',y',z'+t')$, every point $(x,y,z)\in \bd 0 \eps$ with $z>0$, can be reached optimally considering the sub-Riemannian minimizing curve between the origin and $(x,y,0)$ rescaled on time $z$. 
		
		If, instead, $z\le 0$, 
		we cannot construct any sub-Riemannian trajectory from $0$ to $(x,y,z-t)$, $t>0$, with cost  $\le\dsr(0,(x,y,z))$. This is due to the fact that the minimizing trajectories in Heisenberg group are the lifts of arcs on the plane $(x,y)$, spanning area equal to the $z$ coordinates, and that $|z-t|=-z+t>|z|$.
		 Since, by Proposition~\ref{prop:dcont}, $\dd(0,(x,y,z))\le\dsr(0,(x,y,z))$, this implies that there exists no minimizer for $\dd(q,(x,y,z))$.
	\end{example}

\subsection{Estimates on reachable sets}
	
	In this section we apply Theorem~\ref{thm:ballbox}, in order to obtain results in the spirit of Theorem~\ref{thm:srballbox} and Corollary~\ref{cor:holder}.	
	First, we need the following definition.
	
	\begin{defn}
		The point $q$ is said to be \em regular with respect to the drift $\drift$\em, if $q'\mapsto\ord_{q'}(f_0)$ is locally constant at $q$.
	\end{defn}
		
	The main result of this section are the following local regularity estimates for $\dd$. We cannot expect anything global, since in general the sets $\bd q \eps$ are noncompact.

	\begin{thm}\label{thm:holder}
		Assume that \eqref{cs:drift} satisfies the strong H\"ormander condition, and 		
		let $q$ be regular with respect to the drift $\drift$. Assume, moreover, that $z=(z_1,\ldots,z_n)$ is a system of privileged coordinates at $q$ for $\vett$, such that $z_* \drift=\partial_{z_k}$, for some $1\le k \le n$. Then,  there exist $T_0,\eps_0,C>0$ such that 
		\[
			\frac{1} {C} \dist\left(z(q'),z(e^{[0,T]\drift}q)\right) \le \ddt (q,q')\le C   \dist\left(z(q'),z(e^{[0,T]\drift}q)\right) ^{\nicefrac{1}{r}},\quad q'\in \bdt q \eps,
		\]
		where, for any $x\in\real^n$ and $A\subset\real^n$,  $\dist(x,A)=\inf_{y\in A}|x-y|$ is the Euclidean distance between them	and $r$ is the degree of non-holonomy of $\vett$ at $q$.
	\end{thm}

	Let us define the following sets, for parameters $\eta>0$ and $T>0$. We remark that $\bbox \eta$ is defined as in \eqref{def:box} and that $\{\partial_{z_i}\}_{i=1}^n$ is the canonical basis in $\real^n$.
	\[
	\begin{split}
		\Xi_T(\eta) &= \bigcup_{0\le \xi\le T} \bigg(  \xi \partial_{z_k} + \bbox \eta  \bigg), \\ 	
		\Pi_T (\eta) &= \bbox \eta \cup \bigcup_{0<\xi\le T} \{ z\in\real^n\colon\: z_k=\xi,\, |z_i|\le \eta^{w_i} + \eta \xi^{\frac {w_i} s} \text{ for } w_i\le s, i\neq k, \\
		&\qquad \qquad \qquad \qquad \qquad \qquad \qquad \qquad \qquad  \text{ and }|z_i|\le \eta (\eta + \xi^{\frac 1 s})^{w_i-1} \text{ for } w_i> s\},\\
		\widehat\Pi_T (\eta) &= \bbox \eta \cup \bigcup_{0<\xi\le T} \{ z\in\real^n\colon\: z_k=\xi,\, |z_i|\le\eta^{w_i}\text{ for } w_i\le s, i\neq k, \\
		&\qquad \qquad \qquad \qquad \qquad \qquad \qquad \qquad \qquad  \text{ and }|z_i|\le \eta (\eta + \xi^{\frac 1 s})^{w_i-1} \text{ for } w_i> s\}.
	\end{split}
	\]
	
	As in the sub-Riemannian case, Theorem~\ref{thm:holder} is a direct consequence of some estimates on the shape of the accessible sets, contained in the following.
	
	\begin{thm}\label{thm:dballbox}
		Assume that \eqref{cs:drift} satisfies the strong H\"ormander condition, and 		
		let $q\in M$ be regular with respect to the drift $\drift$. Assume, moreover, that $z=(z_1,\ldots,z_n)$ is a system of privileged coordinates at $q$ for $\vett$, such that $z_* \drift=\partial_{z_k}$, for some $1\le k \le n$.
		Then, there exist $C,\eps_0,T_0>0$ such that 
		\begin{equation}\label{eq:bboxdrift}
			 \Xi_T\left( {\frac{1}{C} \eps}\right)\subset  \bdt q \eps \subset \Pi_T(C\eps),\qquad \text{for } \eps<\eps_0\text{ and } T<T_0.
		\end{equation}
		Here, with abuse of notation, we denoted by $\bdt q \eps$ the coordinate representation of the reachable set. In particular,  
		\[
			 \bbox {\frac{1}{C} \eps}\cap \{z_k\le 0\}\subset \bdt q \eps \cap \{z_k\le 0\} \subset \bbox {C\eps}\cap \{z_k\le 0\}.
		\]
		Moreover, if the system is nilpotent, it holds
		\begin{equation}\label{eq:bboxdriftnil}
			 \Xi_T\left( {\frac{1}{C} \eps}\right)\subset  \bdt q \eps \subset \widehat\Pi_T(C\eps),\qquad \text{for } \eps<\eps_0\text{ and } T<T_0.
		\end{equation}
	\end{thm}
	
	In order to prove Theorem~\ref{thm:dballbox}, we need the following lemma.
	
	\begin{lem}\label{lem:time}
		Let $z=(z_1,\ldots,z_n)$ be a system of privileged coordinates at $q\in M$.
		Then there exist $C,\eps_0,T_0>0$ such that,  for any $q'\in\bdt q {\eps_0} $ for $\eps<\eps_0$ and $T<T_0$, and such that
		\begin{enumerate}[(i)]
			\item for any $t<\eps_0$, $\ord_{q'(t)}\drift=-s$, where $q'(t)=e^{-t\drift}(q')$,
			\item $dz_k\big(\drift(z(q'))\big)\neq 0$, for some $k$ with $w_k=-s$,		\end{enumerate}
		 it holds that, if $u\in\lcont$ is a control steering the system \eqref{cs:drift} from $q$ to $q'$, with $\|u\|_1=\eps$, then
		\[
			T\le C \big( \eps^s + \max\{z_k(q'),0\} \big).
		\]
	\end{lem}
	
	\begin{proof}
		For any $\eta>0$, let $\gamma$ be the trajectory associated with $u\in\lcont$ in the system \eqref{cs:drift}, and satisfying  $\gamma:q\rightsquigarrow q'$. 
		Let $\tilde\gamma$ be the trajectory associated with $u$ and starting from $q$, in the time-dependent system \eqref{cs:ad}. Thus $\gamma(t)=e^{t\drift}\circ\tilde\gamma(t)$ and $\dtd(q,\tilde\gamma(T))\le\eps$.
		
		Recall that, for any vector field $g$ and point $p\in M$, it holds that $z_k(e^{Tg}(p))-z_k(p)=\int_0^T  dz_k\big(g(e^{tg}(p))\big)$. Thus, by the mean value theorem, there exists $\tau\in[0,T]$ such that 		
		\begin{equation}\label{eq:zk}
			z_k(q')=z_k(\gamma(T))=T\,dz_k\big(\drift(e^{\tau\drift}(\tilde\gamma(T))) \big)+z_k(\tilde\gamma(T)).
		\end{equation}
		
		Since $e^{\tau\drift}(\tilde\gamma(T))=e^{-(T-\tau)\drift}(q')$, by hypothesis $(ii)$ and the smoothness of $\drift$, there exist $T_0,C_1>0$, independent of $\gamma$, such that $dz_k\big(\drift(e^{\tau\drift}(\tilde\gamma(T))) \big)\ge C_1$ for $T<T_0$.  
		Hence, by Theorem~\ref{thm:ballbox} (since $w_k=s$), there exist $C_2,\bar\eps>0$ such that, if $\eps<\bar\eps$ and $T<T_0$, 
		\[
			T\le \frac{ |z_k(\tilde\gamma(T))|+ \max\{z_k(q'),0\} }{C_1}\le \frac{ C_2 \big( \eps^s+T \eps \big) + \max\{z_k(q'),0\}}{C_1}.
		\]
		Since the constants are independent of $\gamma$, taking $C=C_2/C_1$, $\eps_0\le\min\{T_0,\bar\eps,(C-1)/C^2\}$ completes the proof. 
	\end{proof}

	\begin{proof}[Proof of Theorem~\ref{thm:dballbox}]
		The first inclusion in \eqref{eq:bboxdrift} follows from Theorem~\ref{thm:srballbox} and Theorem~\ref{prop:dcont}. In fact, combining them, we have that, for any $\eps<\eps_0$ and any $T>0$,
		\[
			\Xi_T \left({\frac 1 C \eps}\right)\subset \bigcup_{0\le t \le T}\bsr {e^{t\drift}q} \eps \subset \bdt q \eps.
		\]
		 
		 To prove the second inclusion, we let $q'\in\bdt q \eps$. Fix any $\eta>0$ and consider a control $u\in\lcont$ such that  its associated trajectory $\gamma$, in the system \eqref{cs:drift}, satisfies  $\gamma:q\rightsquigarrow q'$ and $\costd(\gamma)\le\eps+\eta$. We distinguish two cases. First we assume that  $z_k(q')\le 0$. In this case, by Lemma~\ref{lem:time} it follows there exists $C,\eps_0,T_0>0$ such that if $T<T_0$ and $\eps<\eps_0$, then $T\le C\eps^s$. 
		 Moreover \eqref{eq:split} implies that $e^{-T\drift}(q')\in \btdt q \eps T$. Then, enlarging the constant $C$, Theorem~\ref{thm:ballbox} yields
		 \begin{gather*} \label{eq:dzk}
		 	|z_i(q')|=\big|z_i\big(e^{-T\drift}(q')\big)\big|\le C\big(\eps^{w_i}+\eps T^{\frac {w_i} s}\big)\le C \eps^{w_i},\quad \text{if } w_i\le s \text{ and } i\neq k, \\
			\label{eq:dzk3}
			|z_k(q')|\le T + \big|z_k\big(e^{-T\drift}(q')\big)\big|\le T+ C\big( \eps+T^{\frac 1 s})^s\le C\eps^{s},\\
			|z_i(q')|=\big|z_i\big(e^{-T\drift}(q')\big)\big|\le C\eps\big(\eps+ T^{\frac {1} s}\big)^{w_i-1}\le C \eps^{w_i},\quad \text{if } w_i> s.
		 \end{gather*}
		 Here, we used the fact that, for any $p\in M$, from $z_*\drift=\partial_{z_k}$, it holds $z_i(p)= z_i\big(e^{-T\drift}(p)\big)$ and ${|dz_k(\drift(p))|\equiv 1}$.
		 Thus, $q'\subset\bbox {C\eps}\subset\boxd {C\eps}$. 
		 
		 On the other hand, if $z_k(q')>0$, Lemma~\ref{lem:time} yields that $T\le C\big(\eps^s+z_k(q')\big)$. Then, applying again Theorem~\ref{thm:ballbox}, we get
		\[
		 \begin{split}
			 |z_i(q')|&\le C\big(\eps^{w_i}+\eps z_k(q')^{\frac{w_i} s}\big),\quad \text{if } w_i\le s \text{ and } i\neq k, \\			
			 |z_k(q')|&\le T+C\eps^s,\\
			 |z_i(q')|&\le C\eps\big(\eps+ z_k(q')^{\frac {1} s}\big)^{w_i-1},\quad \text{if } w_i> s.
		 \end{split}
		 \]
	This proves that $q'\subset\boxd {C\eps}$, completing the proof of \eqref{eq:bboxdrift}.
	
	To prove \eqref{eq:bboxdriftnil} it suffices to use the same argument as above,  applying the result on nilpotent systems in Theorem~\ref{thm:ballbox}.	 
	\end{proof}

	\begin{rmk}
		Theorem~\ref{thm:dballbox} suggests that the behavior of the system \eqref{cs:drift}, when moving in the direction $-\drift$, is essentially sub-Riemannian. However, although this is true locally in time, it is false in general. For example, consider the Euclidean plane endowed with a rotational drift, i.e., such that $\{e^{t\drift}(q)\}_{t\in(0,+\infty)}$ is diffeomorphic to $\mathcal S^1$ for any $q\neq 0$. Then, $\dd(q,e^{-t\drift}(q))=0$ for any $t>0$ and thus we can move in the direction $-\drift$ for free.
	\end{rmk}
	
	\begin{proof}[Proof of Theorem~\ref{thm:holder}]
		Since every norm on $\real^n$ is equivalent, $\dist(z(q'),[0,T]\partial_{z_k})$ is equivalent to 
		\[
			a(q')=\sum_{\substack{1\le i\le n\\ i\neq k}} |z_i(q')|+\min_{t\in[0,T]} |z_k(q')-t|.
		\]
		Thus, to complete the proof it suffices to prove that it holds $C^{-1}a(q')\le\ddt(q,q')\le Ca(q')^{\nicefrac{1}{r}}$.
		
		By Theorem~\ref{thm:ballbox}, $\Xi_T ({C^{-1}\eps})\subset \bdt q \eps \subset \Pi_T(C\eps)$ for any $\eps<\eps_0$. The first inclusion is equivalent to the fact that, for every $\eps<\eps_0$ such that $Ca(q')\le \eps^r$, one has $\ddt(q,q')\le \eps$. From this follows that $\ddt(q,q')\le C^{\nicefrac{1}{r}}a(q')^{\nicefrac{1}{r}}$.
		The same reasoning applied to the other inclusion proves that 
		\begin{gather*}
			|z_i(q')|\le C(\ddt(q,q')^{w_i}+\ddt(q,q') T^{\frac{w_i}{s}})  \text{ if } w_i\le s,\, i\neq k,\\
			 \min_{t\in[0,T]}|z_k(q')-t|\le C\ddt(q,q')^s,\\
			|z_i(q')|\le C(\ddt(q,q')^{w_i}+\ddt(q,q') T^{\frac{w_i-1}{s}}) \text{ if } w_i > s.
		\end{gather*}
		Clearly, this implies that $a(q')\le C\ddt(q,q')$, for some larger constant, completing the proof of the theorem.
	\end{proof}
	
\bibliographystyle{plain}
\bibliography{ART_HolderContinuity_def}

\end{document}